\RequirePackage{fix-cm}
\documentclass[twocolumn]{svjour3}       

\smartqed  
\usepackage{graphicx}
\usepackage{hyperref}
\usepackage{mathptmx}      
\usepackage{graphicx}
\usepackage{fancyhdr}
\usepackage{amsfonts,amssymb,amsmath,mathtools}
\usepackage{enumitem}
\usepackage[percent]{overpic}
\usepackage{tikz}
\usepackage{mathrsfs}
\usepackage{ifthen,tabularx,graphicx,multirow}
\usepackage{xargs}
\usepackage{tcolorbox}
\usepackage{enumerate}
\usepackage{algorithm}
\usepackage{comment}
\usepackage{algpseudocode}
\usepackage{booktabs}
\usepackage{varwidth}
\usepackage{cite}
\usepackage{overpic}

\makeatletter 
\newcounter{algorithmicH}
\let\oldalgorithmic\algorithmic
\renewcommand{\algorithmic}{
  \stepcounter{algorithmicH}
  \oldalgorithmic}
\renewcommand{\theHALG@line}{ALG@line.\thealgorithmicH.\arabic{ALG@line}}
\makeatother
\newtheorem{assumption}{Assumption}

\journalname{Nonlinear Dynamics}

\graphicspath{ {./../} }
\begin{document}

\title{Beyond expectations: Residual Dynamic Mode Decomposition and Variance for Stochastic Dynamical Systems}


\author{Matthew J. Colbrook         \and
        Qin Li \and
				Ryan V. Raut \and
				Alex Townsend}


\institute{Matthew J. Colbrook \at
DAMTP, University of Cambridge, Cambridge, CB3 0WA, UK\\
\email{m.colbrook@damtp.cam.ac.uk}           
\and
Qin Li \at
Department  of  Mathematics,  University  of  Wisconsin-Madison,  Madison,  WI  53706,  USA\\
\email{qinli@math.wisc.edu}
\and
Ryan V. Raut \at
Allen Institute, Seattle, WA 98109, USA\\
Department of Physiology and Biophysics, University of Washington, Seattle, WA 98195, USA\\
\email{ryan.raut@alleninstitute.org}
\and
Alex Townsend \at
Department of Mathematics, Cornell University, Ithaca, NY 14853, USA\\
\email{townsend@cornell.edu}
}
\date{\today}

\maketitle

\begin{abstract}
Koopman operators linearize nonlinear dynamical systems, making their spectral information of crucial interest. Numerous algorithms have been developed to approximate these spectral properties, and Dynamic Mode Decomposition (DMD) stands out as the poster child of projection-based methods. Although the Koopman operator itself is linear, the fact that it acts in an infinite-dimensional space of observables poses challenges. These include spurious modes, essential spectra, and the verification of Koopman mode decompositions. While recent work has addressed these challenges for deterministic systems, there remains a notable gap in verified DMD methods for stochastic systems, where the Koopman operator measures the expectation of observables. We show that it is necessary to go beyond expectations to address these issues. By incorporating variance into the Koopman framework, we address these challenges. Through an additional DMD-type matrix, we approximate the sum of a squared residual and a variance term, each of which can be approximated individually using batched snapshot data. This allows verified computation of the spectral properties of stochastic Koopman operators, controlling the projection error. We also introduce the concept of variance-pseudospectra to gauge statistical coherency. Finally, we present a suite of convergence results for the spectral information of stochastic Koopman operators. Our study concludes with practical applications using both simulated and experimental data. In neural recordings from awake mice, we demonstrate how variance-pseudospectra can reveal physiologically significant information unavailable to standard expectation-based dynamical models.

\keywords{Dynamical systems \and Koopman operator \and Data-driven discovery \and Dynamic mode decomposition \and Spectral theory \and Error bounds \and Stochastic systems}

\subclass{37M10 \and 37H99 \and 37N25 \and 47A10 \and 47B33 \and 65P99}

\end{abstract}

\section{Introduction}

Stochastic dynamical systems are widely used to model and study systems that evolve under the influence of both deterministic and random effects. They offer a framework for understanding, predicting, and controlling systems exhibiting randomness. This makes them invaluable across various scientific, engineering, and economic applications.

Given a state-space $\Omega\subset \mathbb{R}^d$ and a sample space $\Omega_s$, we consider a discrete-time stochastic dynamical system
\begin{equation}\label{eqn:dynamics}
\pmb{x}_{n} = F(\pmb{x}_{n-1},\tau_n), \qquad n\geq 1, \quad \pmb{x}_n \in\Omega,
\end{equation}
where $\{\tau_n\}_{n\in\mathbb{N}}\in\Omega_s$ are independent and identically distributed (i.i.d.) random variables with distribution $\rho$ supported on $\Omega_s$, $\pmb{x}_0\in\Omega$ is an initial condition, and $F : \Omega\times\Omega_s\rightarrow\Omega$ is a function. In many applications, the function $F$ is unknown or cannot be studied directly, which is the premise of this paper. We adopt the notation $F_{\tau}(\pmb{x})=F(\pmb{x},\tau)$ for convenience and express $\pmb{x}_{n} = (F_{\tau_n}\circ \cdots \circ F_{\tau_1})(\pmb{x}_0)$, where `$\circ$' denotes the composition of functions.

With the assumptions above, equation \eqref{eqn:dynamics} describes a discrete-time Markov process. For such systems, the Kolmogorov backward equation governs the evolution of an observable \cite{kolmogoroff1931analytischen,givon2004extracting}, with the right-hand side defined as the stochastic Koopman operator \cite{mezic2005spectral}. The works \cite{mezic2005spectral,mezic2004comparison} have spurred increased interest in the data-driven approximation of both deterministic and stochastic Koopman operators and in analyzing their spectral properties \cite{mezicAMS,brunton2021modern,kosticlearning}. Prominent applications span a variety of fields including fluid dynamics \cite{schmid2010dynamic,rowley2009spectral,mezic2013analysis,giannakis2018koopman}, epidemiology \cite{proctor2015discovering}, neuroscience \cite{brunton2016extracting,casorso2019dynamic,marrouch2020data}, finance \cite{mann2016dynamic}, robotics \cite{berger2015estimation,bruder2019modeling}, power systems \cite{susuki2011coherent,susuki2011nonlinear}, and molecular dynamics \cite{nuske2014variational,klus2018data,schwantes2015modeling,schwantes2013improvements}.

Although the function $F$ is usually nonlinear, the stochastic Koopman operator is always \textit{linear}; however, it operates on an infinite-dimensional space of observables. Of particular interest is the spectral content of the Koopman operator near the unit circle, which corresponds to slow subspaces encapsulating the long-term dynamics. If finite-dimensional eigenspaces can capture this spectral content effectively, they can serve as a finite-dimensional approximation. Numerous algorithms have been developed to approximate the spectral properties of Koopman operators \cite{budivsic2012applied,mezic2022numerical,mauroy2012use,brunton2017chaos,giannakis2019data,arbabi2017study,das2021reproducing,korda2020data,arbabi2017ergodic,mezic2013analysis}. Among these, Dynamic Mode Decomposition (DMD) is particularly popular \cite{kutz2016dynamic}. Initially introduced in the fluids community \cite{schmid2010dynamic,schmid2009dynamic}, DMD's connection to the Koopman operator was established in \cite{rowley2009spectral}. Since then, several extensions and variants of DMD have been developed \cite{chen2012variants,proctor2016dynamic,baddoo2021physics,colbrook2022mpedmd,williams2015data,williams2015kernel}, including methods tailored for stochastic systems \cite{vcrnjaric2020koopman,sinha2020robust,zhang2022koopman,wanner2022robust}.

At its core, DMD is a projection method. It is widely recognized that achieving convergence and meaningful applications of DMD can be challenging due to the infinite-dimensional nature of Koopman operators \cite{budivsic2012applied,williams2015data,colbrook2021rigorous,kaiser2017data}. Challenges include the presence of spurious (unphysical) modes resulting from projection, essential spectra,\footnote{For an illustrative example of a transition operator with non-trivial essential spectra, refer to \cite{atchade2007geometric}. If the operator in question is either self-adjoint or an $L^2$ isometry, the methodologies described in \cite{colbrook2021computing,SpecSolve} and \cite{colbrook2021rigorous} respectively, can be applied to compute spectral measures.} the absence of non-trivial finite-dimensional invariant subspaces, and the verification of Koopman mode decompositions (KMDs). Residual Dynamic Mode Decomposition (ResDMD) has been introduced to address these issues for deterministic systems \cite{colbrook2021rigorous,colbrook2023residual}. ResDMD facilitates a data-driven approach to compute residuals associated with the full infinite-dimensional Koopman operator, thus enabling the computation of spectral properties with controlled errors and the verification of learned dictionaries and KMDs. Despite the evident importance of analyzing stochastic systems through the Koopman perspective, similar verified DMD methods in this setting are absent.

This paper presents several infinite-dimensional techniques for the data-driven analysis of stochastic systems. The central concept we explore is going beyond expectations to include higher moments within the Koopman framework. Figure \ref{fig:variance motivation} illustrates this point by depicting the evolution of two eigenfunctions associated with the stochastic Van der Pol oscillator (detailed in Section~\ref{sec:stoch_vdp}), alongside the expectation determined by the stochastic Koopman operator. Both eigenvalues and eigenfunctions are computed with a negligible projection error.\footnote{Here, 'projection error' refers to the error incurred when projecting the infinite-dimensional Koopman operator onto a finite-dimensional space of observables.} Notably, although both corresponding eigenvalues oscillate at the same frequency due to having identical arguments, the variances of the trajectories exhibit significant differences. This divergence is quantified by what we define as a \textit{variance residual} (see Section \ref{sec:stochastic_resdmd_def}).

\begin{figure}
\centering
\includegraphics[width=0.4\textwidth]{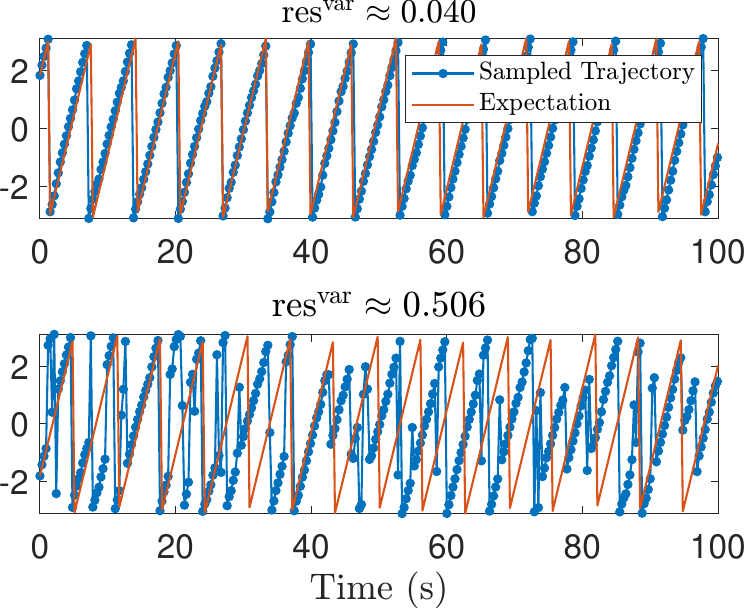}
\caption{The evolution of two eigenfunctions on the attractor of the stochastic Van der Pol oscillator from Section \ref{sec:stoch_vdp}. The plots show the arguments. In blue, we see a sample of the true trajectories, while the expected values predicted from the stochastic Koopman operator are shown in red. Top: Eigenfunction associated with $m=0$ and $k=1$ in Table \ref{tab:SDEtab}. The variance residual is small, and trajectories hug the expectation closely. Bottom: Eigenfunction associated with $m=1$ and $k=1$ in Table \ref{tab:SDEtab}. The variance residual is large, and trajectories deviate from the expectation.}
\label{fig:variance motivation}
\end{figure}

\subsection{Contributions}

The contributions of our paper are as follows:
\begin{itemize}
\item \textbf{Variance Incorporation:} We integrate the concept of variance into the Koopman framework and establish its relationship with batched Koopman operators. Proposition \ref{IMSE_prop} decomposes a mean squared Koopman error into an infinite-dimensional residual and a variance term. Additionally, we present methodologies (see Algorithms \ref{alg:stoch_resDMD1} and \ref{alg:stoch_resDMD2}) for independently calculating these components, thereby enhancing the understanding of the spectral properties of the Koopman operator and the deviation from mean dynamics.
\item \textbf{Variance-Pseudospectra:} We introduce a novel concept of pseudospectra, termed \textit{variance-pseudospectra} (see Definition \ref{def:pseudospectrum_stoch}), which serves as a measure of statistical coherency.\footnote{In the setting of dynamical systems, coherent sets or structures are subsets of the phase space where elements (e.g., particles, agents, etc.) exhibit similar behavior over some time interval. This behavior remains relatively consistent despite potential perturbations or the chaotic nature of the system. In essence, within a coherent structure, the dynamics of elements are closely linked and evolve coherently.} We also offer algorithms for computing these pseudospectra (see Algorithms \ref{alg:res_pseudospec1} and \ref{alg:res_pseudospec2}) and prove their convergence.
\item \textbf{Convergence Theory:} Section \ref{sec:theory} of our paper is dedicated to proving a suite of convergence theorems. These pertain to the spectral properties of stochastic Koopman operators, the accuracy of KMD forecasts, and the derivation of concentration bounds for estimating Koopman matrices from a finite set of snapshot data.
\end{itemize}

Various examples are given in Section \ref{sec:examples} and code is available at: \textcolor[rgb]{0,0,1}{https://github.com/MColbrook/Residual-Dynamic-Mode-Decomposition}.

\subsection{Previous work}

Existing literature on stochastic Koopman operators primarily addresses the challenge of noisy observables in extended dynamic mode decomposition (EDMD) methodologies~\cite{wanner2022robust}, and in techniques for debiasing DMD~\cite{hemati2017biasing,dawson2016characterizing,takeishi2017subspace}. A related concern is the estimation error in Koopman operator approximations due to the finite nature of data sets. This issue is present in both deterministic and stochastic scenarios. As \cite{williams2015data} describes, EDMD converges with large data sets to a Galerkin approximation of the Koopman operator. The work in \cite{mollenhauer2022kernel} thoroughly analyzes kernel autocovariance operators, including nonasymptotic error bounds under classical ergodic and mixing assumptions. In \cite{nuske2023finite}, the authors offer the first comprehensive probabilistic bounds on the finite-data approximation error for truncated Koopman generators in stochastic differential equations (SDEs) and nonlinear control systems. They examine two scenarios: (1) i.i.d. sampling and (2) ergodic sampling, with the latter assuming exponential stability of the Koopman semigroup. Additionally, the variational approach to conformational dynamics (VAC), which bears similarities to DMD, is known for providing spectral estimates of time-reversible processes that result in a self-adjoint transition operator. The connection of VAC with Koopman operators is detailed in \cite{webber2021error}, and the approximation of spectral information with error bounds is discussed in \cite{klus2018data}.

\subsection{Data-driven setup}

We present data-driven methods that utilize a dataset of ``snapshot'' pairs alongside a dictionary of observables. While numerous approaches for selecting a dictionary exist in the literature~\cite{williams2015kernel,williams2015data,coifman2006diffusion,giannakis2012nonlinear,ulam1960collection,vitalini2015basis,wanner2022robust}, this topic is not the primary focus of our current study.\footnote{ResDMD has been shown to effectively verify learned dictionaries in deterministic dynamical systems~\cite{colbrook2023residual}.} Following the methodology outlined in \cite{tu2014dynamic}, we consider our given data to consist of pairs of snapshots, which are
\begin{equation}
\label{eq:trajData}
\texttt{S}=\left\{(\pmb{x}^{(m)},\pmb{y}^{(m)})\right\}_{m=1}^M,\quad \pmb{y}^{(m)}=F(\pmb{x}^{(m)},\tau_m).
\end{equation}
Unlike in deterministic systems, for stochastic systems, it can be beneficial for $\texttt{S}$ to include the same initial condition $\pmb{x}^{(m)}$ multiple times, as each execution of the dynamics yields an independent realization of a trajectory. We say that $\texttt{S}$ is $M_1$-\textit{batched} if it can be split into $M_1$ subsets such that
\begin{align*}
&\texttt{S}=\cup_{j=1}^{M_1}\texttt{S}_j,\\
&\texttt{S}_j=\{(\pmb{x}^{(j)},\pmb{y}^{(j,k)}):k=1,\ldots,M_2,\pmb{y}^{(j,k)}=F(\pmb{x}^{(j)},\tau_{j,k})\}.
\end{align*}
In other words, for each $\pmb{x}^{(j)}$, we have multiple realizations of $F_\tau(\pmb{x}^{(j)})$. Using batched data, we can approximate higher-order stochastic Koopman operators representing the moments of the trajectories. An unbatched dataset can be adapted to approximate a batched dataset by categorizing or ``binning'' the $\pmb{x}$ points in the snapshot data. In practical scenarios, one may encounter a combination of both batched and unbatched data. Depending on the type of snapshot data used, Galerkin approximations of stochastic Koopman operators can be achieved in the limit of large datasets (as discussed in Section \ref{sec:EDMD}).

\section{Mathematical Preliminaries}\label{sec:preliminary}

This section discusses several foundational concepts upon which our paper builds.

\subsection{The stochastic Koopman operator}

Let $g:\Omega\rightarrow \mathbb{C}$ be a function, commonly called an \textit{observable}. Given an initial condition $\pmb{x}_0\in\Omega$, measuring the initial state of the dynamical system through $g$ yields the value $g(\pmb{x}_0)$. One time-step later, the measurement $g(\pmb{x}_1) = g(F_\tau(\pmb{x}_0)) = (g\circ F_\tau)(\pmb{x}_0)$ is obtained, where $\tau$ is a realization from a probability distribution supported on $\Omega_s$, i.e., $\tau\sim\rho$. The ``pull-back'' operator, given $g$, outputs the ``look ahead'' measurement function $g\circ F_\tau$. This function is a random variable, and the stochastic Koopman operator is its expectation \cite{mezic2000comparison}:
\begin{equation}\label{def:K_1}
\mathcal{K}_{(1)}[g] = \mathbb{E}_{\tau}\left[g\circ F_\tau\right]=\int_{\Omega_s} g\circ F_\tau\,\mathrm{d}\rho(\tau).  
\end{equation}
Here, $\mathbb{E}_{\tau}$ represents the expectation with respect to the distribution $\rho$. 
The subscript $(1)$ indicates this is the first moment. Throughout the paper, we assume that the domain of the operator $\mathcal{K}_{(1)}$ is $L^2(\Omega,\omega)$, where $\omega$ is a positive measure on $\Omega$. This space is equipped with an inner product and norm, denoted by $\langle \cdot,\cdot\rangle$ and $\|\cdot\|$, respectively. We do not assume that $\mathcal{K}_{(1)}$ is compact or self-adjoint.

We now introduce the batched Koopman operator, designed to capture the variance and other higher-order moments in the trajectories of dynamical systems. For $r\in\mathbb{N}$ and $g:\Omega^{r}\rightarrow \mathbb{C}$, we define
\begin{equation}\label{eqn:def_K_r}
\mathcal{K}_{(r)}[g] = \mathbb{E}_{\tau}\left[g(F_\tau,\ldots,F_\tau)\right],
\end{equation}
where the same realization $\tau\sim\rho$ is used for the $r$ arguments of $g$. Notably, both the classical and the batched versions of the Koopman operators adhere to the semigroup property, as we will demonstrate.
\begin{proposition}\label{prop:HighK}
For any $r,n\in\mathbb{N}$,
$$
\mathcal{K}_{(r)}^n[g]=\mathbb{E}_{\tau_1,\ldots,\tau_n}\left[g(F_{\tau_n}\circ\cdots \circ F_{\tau_1},\ldots,F_{\tau_n}\circ\cdots \circ F_{\tau_1})\right].
$$
\end{proposition}

\begin{proof}
For $r=1$, see \cite{vcrnjaric2020koopman}. For $r>1$, note that $\mathcal{K}_{(r)}$ is a first-order Koopman operator of a dynamical system on $\Omega^r$.\qed
\end{proof}

This proposition indicates that $n$ applications of the stochastic Koopman operator yield the expected value of an observable after $n$ time steps. It is crucial to understand that $\mathcal{K}_{(1)}$ only calculates the expected value. To gain insights into the variability around this mean and to understand the projection error inherent in DMD methods, we need to consider higher-order statistics, such as the variance. These aspects are further explored in Section \ref{sec:ResDMD_stochastic}.

\subsection{Extended Dynamic Mode Decomposition}
\label{sec:EDMD}

EDMD is a widely-used method for constructing a finite-dimensional approximation of the Koopman operator $\mathcal{K}_{(1)}$, utilizing the snapshot data $\texttt{S}$ in~\eqref{eq:trajData}. This approach involves projecting the infinite-dimensional Koopman operator onto a finite-dimensional matrix and approximating its entries. For notational simplicity, we will omit the subscript $(1)$ when referring to the Koopman operator in this section. Originally, EDMD assumes that the initial conditions are independently drawn from a distribution $\omega$~\cite{williams2015data}. However, in our adaptation, we apply EDMD to any given $\texttt{S}$, treating the $\pmb{x}^{(m)}$ as quadrature nodes for integration with respect to $\omega$. This flexibility allows us to use different quadrature weights depending on the specific scenario.

One first chooses a dictionary $\{\psi_1,\ldots,\psi_{N}\}$ in the space $L^2(\Omega,\omega)$. This dictionary consists of a list of observables that form a finite-dimensional subspace $V_N=\mathrm{span}\{\psi_1,\ldots,\psi_{N}\}$. EDMD computes a matrix $K\in\mathbb{C}^{N\times N}$ that approximates the action of $\mathcal{K}$ within this subspace. Specifically, the goal is to achieve $K=\mathcal{P}_{V_{N}}\mathcal{K}\mathcal{P}_{V_{N}}^*$, where $\mathcal{P}_{V_{N}}:L^2(\Omega,\omega)\rightarrow V_N$ is the orthogonal projection onto $V_N$. In the Galerkin framework, this equates to:
$$
\langle \mathcal{K}[\psi_j],\psi_i\rangle = \sum_{s=1}^N K_{s,j}\langle \psi_s,\psi_i\rangle, \qquad 1\leq i,j\leq N. 
$$
A matrix $K$ satisfying this relationship is given by
$$ 
K = G^{\dagger}A, \qquad G_{i,j} = \langle\psi_j\,,\psi_i\rangle\,,\quad A_{i,j} = \langle\mathcal{K}[\psi_j]\,,\psi_i\rangle\,.
$$
Commonly, we stack the $\Psi$ and define the feature map
$$
\Psi(\pmb{x})=\begin{bmatrix}\psi_1(\pmb{x}) & \cdots& \psi_N(\pmb{x}) \end{bmatrix}^\top\in\mathbb{C}^{1\times N}\,.
$$
Then, for any $g\in V_N$, we use the shorthand $g=\Psi\pmb{g}$ for $g(\pmb{x}) = \sum_{j=1}^N g_j\psi_j(\pmb{x})$. With the previously defined $K$, the approximation becomes
$$
\mathcal{K}[g](\pmb{x}) \approx \sum_{i=1}^N  \left(\sum_{j=1}^N K_{i,j}g_j\right) \psi_i(\pmb{x})=\Psi(\pmb{x})K\pmb{g}.
$$
The accuracy of this approximation depends on how well $V_N$ can approximate $\mathcal{K}g$.

The entries of the matrices $G$ and $A$ are inner products and must be approximated using the trajectory data $\texttt{S}$. For quadrature weights $\{w_m\}$, we define $\tilde{G}$ as the numerical approximation of $G$:
\begin{equation} 
\tilde{G}_{i,j} = \sum_{m=1}^{M} w_{m} \psi_j(\pmb{x}^{(m)})\overline{\psi_i(\pmb{x}^{(m)})}\approx \langle\psi_j\,,\psi_i\rangle\,= {G}_{i,j}\,.\label{eq:ApproximationOfG}
\end{equation}
The weights $\{w_m\}$ reflect the significance assigned to each snapshot in the dataset, influenced by factors such as data distribution or reliability, which we will explore further. Similarly, for $A$, we define
\begin{equation} 
\tilde{A}_{i,j} = \sum_{m=1}^{M} w_{m} \psi_j(\pmb{y}^{(m)})\overline{\psi_i(\pmb{x}^{(m)})} \approx \langle\mathcal{K}[\psi_j]\,,\psi_i\rangle\,=A_{i,j}\,.
\label{eq:ApproximationOfA}
\end{equation} 
Let $\Psi_X,\Psi_Y\in\mathbb{C}^{M\times N}$ collect the dictionary's evaluations of these samples:
\begin{equation}
\label{PSI_defs}
\Psi_X=\begin{pmatrix}
\Psi^\top(\pmb{x}^{(1)})\\
\vdots\\
\Psi^\top(\pmb{x}^{(M)})
\end{pmatrix}\,,\quad \Psi_Y=\begin{pmatrix}
\Psi^\top(\pmb{y}^{(1)})\\
\vdots\\
\Psi^\top(\pmb{y}^{(M)})
\end{pmatrix}\,,
\end{equation}
and let $W=\mathrm{diag}(w_1,\ldots,w_{M})$. Then we can succinctly write
\begin{equation}\label{eqn:def_AG_tilde}
\tilde{G}=\Psi_X^*W\Psi_X,\quad \tilde{A}=\Psi_X^*W\Psi_Y.
\end{equation}
Throughout this paper, the symbol $\tilde{X}$ denotes an estimation of the quantity $X$. 

Various sampling methods converge in the large data limit, meaning that
\begin{equation}
\label{quad_convergence}
\lim_{M\rightarrow\infty} \tilde{G}=G,\quad \lim_{M\rightarrow\infty} \tilde{A}=A.
\end{equation}
We detail three convergent sampling methods:
\begin{itemize}
	\item[(i)] {\textbf{Random sampling:}} In the initial definition of EDMD, $\omega$ is a probability measure and $\{\pmb{x}^{(m)}\}_{m=1}^M$ are independently drawn according to $\omega$ with each quadrature weight set to $w_m=1/M$. The strong law of large numbers guarantees that~\eqref{quad_convergence} holds with probability one~\cite[Section 3.4]{2158-2491_2016_1_51}\cite[Section 4]{korda2018convergence}. Typically, convergence occurs at a Monte Carlo rate of $\mathcal{O}(M^{-1/2})$~\cite{caflisch1998monte}.
	\item[(ii)] {\textbf{Ergodic sampling:}} If the stochastic dynamical system is ergodic, the Birkhoff--Khinchin theorem \cite[Theorem II.8.1, Corollary 3]{gikhman2004theory} supports convergence using data from a single trajectory for almost every initial point. Specifically, we use:
$$
\pmb{x}^{(m+1)}=F(\pmb{x}^{(m)},\tau_{m}),\quad w_m=1/M.
$$
This sampling method's analysis for stochastic Koopman operators is detailed in~\cite{wanner2022robust}. An advantage is that knowledge of $\omega$ is not required. However, the convergence rate depends on the specific problem~\cite{kachurovskii1996rate}. Note that in an ergodic system, the stochastic Koopman operator is an isometry on $L^1(\Omega,\omega)$ but typically not on $L^2(\Omega,\omega)$.
	\item[(iii)] {\textbf{High-order quadrature:}} When the dictionary and $F$ are sufficiently regular, and the dimension $d$ is not too large, and if we can choose the $\{\pmb{x}^{(m)}\}_{m=1}^{M}$
 , employing a high-order quadrature rule is advantageous. For deterministic systems, this approach can significantly increase convergence rates in~\eqref{quad_convergence}~\cite{colbrook2021rigorous}. In stochastic systems, high-order quadrature applies primarily to batched snapshot data. We may select $\{\pmb{x}^{(j)}\}_{j=1}^{M_1}$ based on an $M_1$-point quadrature rule with associated weights $\{w_j\}_{j=1}^{M_1}$. Convergence is achieved as $M_2\rightarrow\infty$, effectively applying Monte Carlo integration of the random variable $\tau$ over $\Omega_s$ for each fixed $\pmb{x}^{(j)}$.	
\end{itemize}

The convergence described in \eqref{quad_convergence} implies that the eigenvalues obtained through EDMD converge to the spectrum of $\mathcal{P}_{V_{N}}\mathcal{K}\mathcal{P}_{V_{N}}^*$ as $M\rightarrow\infty$. Therefore, approximating the spectrum of $\mathcal{K}$, denoted $\mathrm{Sp}(\mathcal{K})$, by the eigenvalues of $\tilde{K}$ is closely related to the so-called finite section method~\cite{bottcher1983finite}. However, just as the finite section method can be prone to spectral pollution, which refers to the appearance of spurious modes that accumulate even as the size of the dictionary increases, this is also a concern for EDMD~\cite{williams2015data}. Consequently, having a method to validate the accuracy of the proposed eigenvalue-eigenvector pairs becomes crucial, which is one of the key functions of ResDMD.

\subsection{Residual Dynamic Mode Decomposition (ResDMD)}\label{sec:resdmd_recap}

Accurately estimating the spectrum of $\mathcal{K}$ is critical for analyzing dynamical systems. For deterministic systems, ResDMD achieves this goal, providing robust spectral estimates~\cite{colbrook2021rigorous,colbrook2023residual}. Unlike classical DMD methods, ResDMD introduces an additional matrix specifically designed to approximate  $\mathcal{K}^*\mathcal{K}$. This enhancement not only offers rigorous error guarantees for the spectral approximation but also enables a posteriori assessment of the reliability of the computed spectra and Koopman modes. This capability is particularly valuable in addressing issues such as spectral pollution, which are common challenges in DMD-type methods.

ResDMD is built around the approximation of residuals associated with $\mathcal{K}$, providing an error bound. For any given candidate eigenvalue-eigenvector pair $(\lambda,g)$, with $\lambda\in\mathbb{C}$ and $g=\Psi\,\pmb{g}\in V_{N}$, one can consider the relative squared residual as follows:
\begin{align}
\label{residual_form1_OLD}
&\frac{\int_{\Omega}\left|\mathcal{K}[g](\pmb{x})-\lambda g(\pmb{x})\right|^2\,\mathrm{d}\omega(\pmb{x})}{\int_{\Omega}\left|g(\pmb{x})\right|^2\,\mathrm{d}\omega(\pmb{x})}\\
&=\frac{\langle \mathcal{K}[g],\mathcal{K}[g]\rangle -\lambda\langle g,\mathcal{K}[g]\rangle -\overline{\lambda}\langle \mathcal{K}[g],g\rangle+|\lambda|^2\langle g,g\rangle}{\langle g,g\rangle}.\notag
\end{align}
This pair $(\lambda,g)$ can be computed either from $K$ or other methods. A small residual means that $\lambda$ can be approximately considered as an eigenvalue of $\mathcal{K}$, with $g$ as the corresponding eigenfunction. The relative residual in~\eqref{residual_form1_OLD} serves as a measure of the \textit{coherency} of observables, indicating that observables with smaller residuals play a significant role in the dynamics of the system. If the relative (non-squared) residual is bounded by $\epsilon$, then $\mathcal{K}^ng=\lambda^n g+\mathcal{O}(n\epsilon)$. In other words, $\lambda$ characterizes the coherent oscillation and the decay/growth in the observable  $g$ with time.

The residual is closely related to the notion of pseudospectra~\cite{trefethen2005spectra}.
\begin{definition}\label{def:spectrum_deterministic}
\textit{For any $\lambda\in\mathbb{C}$, define:
$$
\sigma_{\mathrm{inf}}(\lambda)=\inf\left\{\|\mathcal{K}[g]-\lambda g\|:g{\in}L^2(\Omega,\omega),\|g\|=1\!\right\}.
$$
For $\epsilon>0$, the approximate point\footnote{In the presence of residual spectrum, the full pseudospectrum requires the injection modulus of complex shifts of the adjoint of $\mathcal{K}$. We have refrained from this discussion for the sake of simplicity.} $\epsilon$-pseudospectrum is
$$
\mathrm{Sp}_{\epsilon}(\mathcal{K})=\mathrm{Cl}\left(\left\{\lambda\in\mathbb{C}:\sigma_{\mathrm{inf}}(\lambda)<\epsilon\right\}\right),
$$
where $\mathrm{Cl}$ denotes closure of a set. Furthermore, we say that $g$ is a $\epsilon$-pseudoeigenfunction if there exists $\lambda\in\mathbb{C}$ such that the relative squared residual in~\eqref{residual_form1_OLD} is bounded by $\epsilon^2$.}
\end{definition}

To compute~\eqref{residual_form1_OLD}, notice that three of the four inner products appearing in the numerator are:
\begin{equation}\label{eqn:def_AG}
\langle \mathcal{K}[g],g\rangle=\pmb{g}^*A\pmb{g},\;\langle g,\mathcal{K}[g]\rangle=\pmb{g}^*A^*\pmb{g},\;  \langle g,g\rangle=\pmb{g}^*G\pmb{g},
\end{equation}
with $A,G$ numerically approximated by EDMD~\eqref{eqn:def_AG_tilde}. Hence, the success of the computation relies on finding a numerical approximation to $\langle \mathcal{K}[g],\mathcal{K}[g]\rangle$. To that end, we deploy the same quadrature rule discussed in~\eqref{eq:ApproximationOfG}-\eqref{eq:ApproximationOfA} and set
\begin{equation}\label{eqn:def_L}
L=[L_{i,j}]\,,\quad L_{i,j} = \langle\mathcal{K}[\psi_j],\mathcal{K}[\psi_i]\rangle,\quad \tilde{L}=\Psi_Y^*W\Psi_Y\,,
\end{equation}
then $\langle \mathcal{K}[g],\mathcal{K}[g]\rangle\approx \pmb{g}^*\Psi_Y^*W\Psi_Y\pmb{g}=\pmb{g}^*\tilde{L}\pmb{g}$. We obtain a numerical approximation of~\eqref{residual_form1_OLD} as
\begin{equation}
\label{ResDMD_explanation}
[\mathrm{res}(\lambda,g)]^2=\frac{\pmb{g}^*\left[\tilde{L}- \lambda\tilde{A}^* - \overline{\lambda}\tilde{A} + |\lambda|^2\tilde{G}\right]\pmb{g}}{\pmb{g}^*\tilde{G}\pmb{g}}. 
\end{equation}
The matrix $L$ introduced by ResDMD formally corresponds to an approximation of $\mathcal{K}^*\mathcal{K}$. The computation utilizes the same dataset as that employed for $\tilde{G}$ and $\tilde{A}$ and is computationally efficient to construct. The work presented in~\cite{colbrook2021rigorous} demonstrates that the approximation outlined in~\eqref{ResDMD_explanation} can be effectively used in various algorithms for rigorously computing the spectra and pseudospectra of $\mathcal{K}$ for deterministic systems. However, these results from~\cite{colbrook2021rigorous} are not directly applicable to stochastic systems.

\section{Variance from the Koopman perspective}\label{sec:ResDMD_stochastic}

When analyzing a system with inherent stochasticity, basing conclusions only on the mean trajectory can lead to misleading interpretations, as illustrated in Figure \ref{fig:variance motivation}. To achieve a more accurate statistical understanding of such systems, it is crucial to quantify how much and in what ways the trajectory deviates from this mean. This need for a more comprehensive analysis underpins our exploration into quantifying the variance.

\subsection{Variance via Koopman operators}
\label{sec:var}
For any observable $g\in L^2(\Omega,\omega)$ and $\pmb{x}\in\Omega$, $g(F_\tau(\pmb{x}))$ is a random variable. One can define its moments:
$$
\mathbb{E}_{\tau}[(g(F_\tau(\pmb{x})))^r]=\int_{\Omega_s} [g(F_\tau(\pmb{x}))]^r\,\mathrm{d}\rho(\tau),\quad r\in\mathbb{N}.
$$
Recalling the definitions in~\eqref{eqn:def_K_r}, this becomes:
$$
\mathbb{E}_{\tau}[(g(F_\tau(\pmb{x})))^r]=\mathcal{K}_{(r)}[g\otimes \cdots \otimes g](\pmb{x},\ldots,\pmb{x})\,.
$$
This means that the $r$-th order Koopman operator directly computes the moments of the trajectory. In particular, the combination of the first and the second moment provides the following variance term:
\begin{align*}
\text{Var}_{\tau}[g(F_\tau(\pmb{x}))]&= \mathbb{E}_\tau\left[|g(F_\tau(\pmb{x}))|^2\right]-|\mathbb{E}_\tau[g(F_\tau(\pmb{x}))]|^2\\
&= \mathcal{K}_{(2)}[g\otimes\overline{g}](\pmb{x},\pmb{x})-|\mathcal{K}_{(1)}[g](\pmb{x})|^2\,.
\label{eq:VarianceKoopman} 
\end{align*}
We integrate the local definition of variance over the entire domain to define:
\begin{align}
\text{Var}_{\tau}[g(F_\tau)]&= \int_\Omega\text{Var}_{\tau}[g(F_\tau(\pmb{x})]\,\mathrm{d}\omega(\pmb{x}).
\end{align}

The following proposition provides a Koopman analog of decomposing an integrated mean squared error (IMSE).
\begin{proposition}\label{IMSE_prop}
Let $g,h\in L^2(\Omega,\omega)$, then
\begin{equation}
\label{IMSE_decomp}
\begin{split}
&\mathbb{E}_{\tau}\left[\|g\circ F_\tau +h\|^2\right]\\
&=\|\mathcal{K}_{(1)}[g]+h\|^2+\int_{\Omega}\mathrm{Var}_{\tau}\left[\left(g\circ F_\tau\right)(\pmb{x})\right]\,\mathrm{d}\omega(\pmb{x}).
\end{split}
\end{equation}
\end{proposition}
\begin{proof}
We expand $|g(F_\tau(\pmb{x}))+h(\pmb{x})|^2$ for a fixed $\pmb{x}\in\Omega$ and take expectations to find that
\begin{align*}
&\mathbb{E}_{\tau}\left[ |g(F_\tau(\pmb{x}))+h(\pmb{x})|^2\right]\\
&{=}\mathbb{E}_{\tau}\left[|g(F_\tau(\pmb{x}))|^2\right]{+}\mathcal{K}_{(1)}[g](\pmb{x})\overline{h(\pmb{x})}{+}h(\pmb{x})\overline{\mathcal{K}_{(1)}[g](\pmb{x})}{+}|h(\pmb{x})|^2\\
&{=}|\mathcal{K}_{(1)}[g](\pmb{x})+h(\pmb{x})|^2+\mathbb{E}_{\tau}\left[|g(F_\tau(\pmb{x}))|^2\right]-\left|\mathbb{E}_{\tau}\left[g(F_\tau(\pmb{x}))\right]\right|^2.
\end{align*}
The result now follows by integrating over $\pmb{x}$ with respect to the measure $\omega$.\qed
\end{proof}

Similarly, for any two functions $g,h\in L^2(\Omega,\omega)$, we define the covariance:
\begin{equation}
\label{cov_mat_def}
\mathcal{C}(g,h){=}\int_{\Omega}\mathbb{E}_{\tau}[(g\circ F_\tau{-}\mathcal{K}_{(1)}[g])\overline{(h\circ F_\tau{-}\mathcal{K}_{(1)}[h])}]\,\mathrm{d}\omega(\pmb{x})
\end{equation}
and obtain the following similar result using covariance:
\begin{align*}
\int_{\Omega} \mathbb{E}_{\tau}[g(F_\tau(\pmb{x}))\overline{h(F_\tau(\pmb{x}))}] \,\mathrm{d} \omega(\pmb{x})
=\langle \mathcal{K}[g],\mathcal{K}[h] \rangle+ \mathcal{C}(g,h)\,.
\end{align*}
Proposition \ref{IMSE_prop} is analogous to the decomposition of an IMSE and is practically useful. Suppose we use an observation $h$ to approximate $-g\circ F_\tau$, in an attempt to minimize $\|g\circ F_\tau +h\|^2$. An unbiased estimator is $-\mathcal{K}_{(1)}[g]$; however, this approximation will not be perfect due to the variance term in~\eqref{IMSE_decomp}. Therefore, there is a variance-residual tradeoff for stochastic Koopman operators. Depending on the type of trajectory data collected, one can approximate the quantities $\mathbb{E}_{\tau}\left[\|g\circ F_\tau +h\|^2\right]$ and $\|\mathcal{K}_{(1)}[g]+h\|^2$  in~\eqref{IMSE_decomp} and hence, estimate the third variance term.

\begin{example}[Circle map]
\label{example:circle_map}
\textit{Let $\Omega=[0,1]_{\mathrm{per}}$ be the periodic interval and consider
$$
F(\pmb{x},\tau)=\pmb{x}+c+f(\pmb{x})+\tau \,\,\,\,\,\,\mathrm{mod}(1),
$$
where $\Omega_s=[0,1]_{\mathrm{per}}$, $\rho$ is absolutely continuous, and $c$ is a constant. Let $\psi_j(\pmb{x})=e^{2\pi i j\pmb{x}}$ for $j\in\mathbb{Z}$. Then
\begin{align}
\mathcal{K}_{(1)}[\psi_j](\pmb{x})=\psi_j(\pmb{x})e^{2\pi i jf(\pmb{x})}e^{2\pi i jc}\int_{\Omega_s} e^{2\pi i j\tau}\,\mathrm{d}\rho(\tau).
\end{align}
Define the constants
$$
\alpha_j=e^{2\pi i jc}\int_{\Omega_s} e^{2\pi i j\tau}\,\mathrm{d}\rho(\tau).
$$
Let $D$ be the operator that multiplies each $\psi_j$ by $\alpha_j$. Then $\mathcal{K}_{(1)}=T D$, where $T$ is the Koopman operator corresponding to $\pmb{x}\mapsto\pmb{x}+f(\pmb{x})$. Since $\rho$ is absolutely continuous, the Riemann--Lebesgue lemma implies that $\lim_{|j|\rightarrow\infty}\alpha_j=0$ and hence $D$ is a compact operator. It follows that if $T$ is bounded, then $\mathcal{K}_{(1)}$ is a compact operator. A straightforward computation using \eqref{eq:VarianceKoopman} shows that
\begin{equation}
\label{circle_var_no_f}
\int_{\Omega}\text{\rm Var}_{\tau}[\psi_j(F_\tau(\pmb{x}))]\,\mathrm{d}\omega(\pmb{x}) = 1-|\alpha_j|^2.
\end{equation}
For example, if $f=0$, $\mathcal{K}_{(1)}$ has pure point spectrum with eigenfunctions $\psi_j$. However, as $|j|\rightarrow \infty$, the variance converges to one and $\psi_j$ become less statistically coherent. This example is explored further in Section \ref{sec:cat_map}.}\qed
\end{example}

Another immediate application of the variance term is in providing an estimated bound for the Koopman operator prediction of trajectories.
\begin{proposition}\label{Prop:chernoff}
We have
\begin{equation}
\label{chernoff_bound}
\begin{split}
&\mathbb{P}\left(\left|g\circ F_{\tau_n}\circ\cdots\circ F_{\tau_1}(\pmb{x})-\mathcal{K}^n[g](\pmb{x})\right|\geq a\right)\\
&\quad\quad\quad\quad\leq\frac{1}{a^2}\text{Var}_{\tau_1,\ldots,\tau_n} \left[g\circ F_{\tau_n}\circ\cdots\circ F_{\tau_1}(\pmb{x})\right]\\
&\quad\quad\quad\quad=\frac{1}{a^2}\left(\mathcal{K}_{(2)}^n[g\otimes\overline{g}](\pmb{x},\pmb{x})-|\mathcal{K}_{(1)}^n[g](\pmb{x})|^2\right)
\end{split}
\end{equation}
for any $a>0$.
\end{proposition}
\begin{proof}
the result follows from combining Proposition \ref{prop:HighK} and \eqref{eq:VarianceKoopman} with Chernoff's bound.\qed
\end{proof}
The bound can be combined with concentration bounds for $\Psi \tilde{K}^n-\mathcal{K}^n$ (see Section \ref{sec:forecast_bound}). 

\subsection{ResDMD in stochastic systems}
\label{sec:stochastic_resdmd_def}

In the deterministic setting, ResDMD provides an efficient way to evaluate the accuracy of candidate eigenpairs through the computation of an additional matrix $L$ in~\eqref{eqn:def_L}. However, what happens in the stochastic setting?

Suppose that $(\lambda, g)$ is a candidate eigenpair of $\mathcal{K}_{(1)}$ with $g\in V_N$. Resembling~\eqref{residual_form1_OLD}, we consider
\begin{equation}
\label{resdmd_stoch_res_res}
\frac{\mathbb{E}_{\tau}\left[\|g\circ F_\tau-\lambda g\|^2\right]}{\|g\|^2}.
\end{equation}
We can write the numerator in terms of $A$, $G$, and $L$, i.e., 
\begin{align*}
\mathbb{E}_{\tau}\left[\|g\circ F_\tau-\lambda g\|^2\right]&=\pmb{g}^*(L-\lambda A^*-\overline{\lambda}A+|\lambda|^2G)\pmb{g}\\
&=\lim_{M\rightarrow\infty}\pmb{g}^*(\tilde{L}-\lambda \tilde{A}^*-\overline{\lambda}\tilde{A}+|\lambda|^2\tilde{G})\pmb{g}.
\end{align*}
Hence, we define
\begin{equation}
\label{eq:abs_res}
[\mathrm{res}^{\mathrm{var}}(\lambda,g)]^2=\frac{\pmb{g}^*\left[\tilde{L}-\lambda \tilde{A}^*-\overline{\lambda}\tilde{A}+|\lambda|^2\tilde{G}\right]\pmb{g}}{\pmb{g}^*\tilde{G}\pmb{g}},
\end{equation}
which furnishes an approximation of~\eqref{resdmd_stoch_res_res}. Setting $h=-\lambda g$ in Proposition~\ref{IMSE_prop}, we see that
\begin{equation} 
\begin{aligned}
&\mathbb{E}_{\tau}\left[\|g\circ F_\tau-\lambda g\|^2\right]=\mathbb{E}_{\tau}\left[\int_{\Omega}|g(F_\tau(\pmb{x}))-\lambda g(\pmb{x})|^2\,\mathrm{d}\omega(\pmb{x})\right]\\
&\quad\quad=\underbrace{\|\mathcal{K}_{(1)}[g]-\lambda g\|^2}_{\text{squared residual}} +\underbrace{\int_{\Omega}\mathrm{Var}_{\tau}\left[g(F_\tau(\pmb{x}))\right]\,\mathrm{d}\omega(\pmb{x})}_{\text{integrated variance of $g\circ F_\tau$}}.
\end{aligned}
\label{eq:ExpectationVarianceRelation} 
\end{equation} 
Thus, $\mathrm{res}^{\mathrm{var}}(\lambda,g)$ approximates the sum of the squared residual $\|\mathcal{K}[g]-\lambda g\|^2$ and the integrated variance of $g\circ F_{\tau}$. For stochastic systems, the integrated variance of $g\circ F_\tau$ is usually non-zero so that
\begin{equation}\label{eq:ResidualFormula}
\lim_{M\rightarrow\infty}\mathrm{res}^{\mathrm{var}}(\lambda,g)> \|\mathcal{K}_{(1)}[g]-\lambda g\|\|g\|.
\end{equation}

Based on this notion and drawing an analogy with Definition~\ref{def:spectrum_deterministic}, we make the following definition.
\begin{definition}\label{def:pseudospectrum_stoch}
\textit{For any $\lambda\in\mathbb{C}$, define:
$$
\sigma_{\mathrm{inf}}^{\mathrm{var}}(\lambda)=\inf\left\{\sqrt{\mathbb{E}_{\tau}\!\left[\|g\circ F_\tau{-}\lambda g\|^2\right]}:g{\in}L^2(\Omega,\omega),\|g\|=1\!\right\}.
$$
For $\epsilon>0$, we define the variance-$\epsilon$-pseudospectrum as
$$
\mathrm{Sp}_{\epsilon}^{\mathrm{var}}(\mathcal{K}_{(1)})=\mathrm{Cl}\left(\left\{\lambda\in\mathbb{C}:\sigma_{\mathrm{inf}}^{\mathrm{var}}(\lambda)<\epsilon\right\}\right),
$$
where $\mathrm{Cl}$ denotes the closure of a set. Furthermore, we say that $g$ is a variance-$\epsilon$-pseudoeigenfunction if there exists $\lambda\in\mathbb{C}$ such that
$\sqrt{\mathbb{E}_{\tau}\!\left[\|g\circ F_\tau{-}\lambda g\|^2\right]}\leq\epsilon$.}
\end{definition}

Superficially, this definition is a straightforward extension of Definition~\ref{def:spectrum_deterministic}. However, there are some essential differences. Both the conceptual understanding and the computation methods need to be modified.

First, the relation~\eqref{eq:ExpectationVarianceRelation} shows that $\mathrm{Sp}_{\epsilon}^{\mathrm{var}}(\mathcal{K}_{(1)})$ takes into account uncertainty through the variance term. Hence, the variance-pseudospectrum provides a notion of \textit{statistical coherency}. Furthermore, comparing Definition~\ref{def:spectrum_deterministic} and Definition~\ref{def:pseudospectrum_stoch}, we have
$$
\mathrm{Sp}_{\epsilon}^{\mathrm{var}}(\mathcal{K}_{(1)})\subset \mathrm{Sp}_{\epsilon}(\mathcal{K}_{(1)})\,.
$$
If the dynamical system is deterministic, then $\mathrm{Sp}_{\epsilon}^{\mathrm{var}}(\mathcal{K}_{(1)})$ is equal to the approximate point $\epsilon$-pseudospectrum. However, in the presence of variance, they are no longer equal.

Second, the relation~\eqref{eq:ExpectationVarianceRelation} gives a computational surprise. Following the same derivation between~\eqref{residual_form1_OLD}-\eqref{ResDMD_explanation}, with $L$, $A$, and $G$ accordingly adjusted through replacing $\mathcal{K}$ by $\mathcal{K}_{(1)}$ in~\eqref{eqn:def_AG}-\eqref{eqn:def_L}, we can still compute the variance-residual term. However, the original residual itself, $\mathrm{res}(\lambda,g)$, needs a modification. Recalling~\eqref{residual_form1_OLD}, in the same spirit of EDMD, if $g\in V_N$, we write
\begin{align*}
&\|\mathcal{K}_{(1)}[g]-\lambda g\|^2\\
&\quad\quad= \langle\mathcal{K}_{(1)}[g]\,,\mathcal{K}_{(1)}[g]\rangle -\lambda \langle g,\mathcal{K}_{(1)}[g]\rangle \\
&\quad\quad\quad\quad\quad\quad-\bar{\lambda} \langle \mathcal{K}_{(1)}[g],g\rangle + |\lambda|^2\langle g,g\rangle \\
&\quad\quad= \pmb{g}^*({H}-\lambda {A}^*-\overline{\lambda}{A}+|\lambda|^2{G})\pmb{g},
\end{align*}
where $H$ is a newly introduced matrix with
\begin{equation}\label{eqn:new_H_matrix}
H_{i,j}=\langle \mathcal{K}_{(1)}[\psi_j],\mathcal{K}_{(1)}[\psi_i] \rangle.
\end{equation}
We employ the quadrature rule for the $\pmb{x}$-domain to approximate this new term. If $\texttt{S}$ is batched with $M_2=2$, then we can form the matrix
$$
\tilde{H}_{i,j}=\sum_{l=1}^{M_1} w_{l} \psi_j(\pmb{y}^{(l,1)})\overline{\psi_i(\pmb{y}^{(l,2)})}.
$$
Since $\tau_{l,1}$ and $\tau_{l,2}$ are independent, we have
\begin{equation}
\label{equ:new_matrix_H}
\lim_{M_1\rightarrow\infty} \tilde{H}_{i,j}=H_{i,j}=\langle \mathcal{K}[\psi_j],\mathcal{K}[\psi_i] \rangle.
\end{equation}
We stress that $\mathcal{K}_{(1)}$ is applied separately to $\psi_i$ and $\psi_j$ and thus $\tau_{l,1}$ and $\tau_{l,2}$ need to be independent realizations.

The convergence in \eqref{equ:new_matrix_H} allows us to compute the spectral properties of $\mathcal{K}_{(1)}$ directly (see Section \ref{sec:algorithms}). In particular, instead of \eqref{ResDMD_explanation}, we now have
\begin{equation}
\label{eq:abs_res2}
[\mathrm{res}(\lambda,g)]^2=\frac{\pmb{g}^*\left[\tilde{H}-\lambda \tilde{A}^*-\overline{\lambda}\tilde{A}+|\lambda|^2\tilde{G}\right]\pmb{g}}{\pmb{g}^*\tilde{G}\pmb{g}}
\end{equation}
and the approximate decomposition
\begin{equation}
\label{eq:res_decomp}
\begin{aligned}
&\int_{\Omega}\mathrm{Var}_{\tau}\left[g(F_\tau(\pmb{x}))\right]\,\mathrm{d}\omega(\pmb{x})=\pmb{g}^*\left(L-H\right)\pmb{g}\\
&\quad\approx \pmb{g}^*\!\left(\tilde{L}{-}\tilde{H}\right)\!\pmb{g}=\|g\|^2\!\left([\mathrm{res}^{\mathrm{var}}(\lambda,g)]^2{-}[\mathrm{res}(\lambda,g)]^2\right)\!,\!
\end{aligned}
\end{equation}
which becomes exact in the large data limit. 

\subsection{Algorithms}\label{sec:algorithms}
In the derivations above, we noticed that one-batched data permits computation only of $\mathrm{res}^\mathrm{var}(\lambda,g)$, while two-batched data also permits the computation of $\mathrm{res}(\lambda,g)$. Algorithms \ref{alg:stoch_resDMD1} and \ref{alg:stoch_resDMD2} approximate the relative residuals of EDMD eigenpairs in the scenario of unbatched and batched data, respectively. In Algorithm \ref{alg:stoch_resDMD2}, we have taken an average when computing $\tilde{A}$ and $\tilde{L}$ to reduce quadrature error, and an average when computing $\tilde{H}$ to ensure that it is self-adjoint (and positive semi-definite). Algorithm \ref{alg:res_pseudospec1} approximates the pseudospectrum and corresponding pseudoeigenfunctions, given batched snapshot data. Algorithm \ref{alg:res_pseudospec2} approximates the variance-pseudospectrum and corresponding variance-pseudoeigenfunctions, and does not need batched data. Note that the computational complexity of all of these algorithms scales the same as those for ResDMD, which is discussed in \cite{colbrook2021rigorous,colbrook2023residual}. In particular, Algorithms \ref{alg:stoch_resDMD1} and \ref{alg:stoch_resDMD2} scale the same as EDMD.

\begin{algorithm}[t]
\textbf{Input:} Snapshot data $\{\pmb{x}^{(m)}\}_{m=1}^{M},\{\pmb{y}^{(m)}\}_{m=1}^{M}$ ($\pmb{y}^{(m)}=F(\pmb{x}^{(m)},\tau_m)$), quadrature weights $\{w_m\}_{m=1}^{M}$, and dictionary of observables $\{\psi_j\}_{j=1}^{N}$.\\
\vspace{-4mm}
\begin{algorithmic}[1]
\State Compute
$$
\tilde{G}=\Psi_X^*W\Psi_X,\quad
\tilde{A}=\Psi_X^*W\Psi_Y,\quad
\tilde{L}=\Psi_Y^*W\Psi_Y,
$$
where $\Psi_X$ and $\Psi_Y$ are given in~\eqref{PSI_defs}. 
\State Solve $\tilde{A}\pmb{g}=\lambda \tilde{G}\pmb{g}$ for eigenpairs $\{(\lambda_j,g_{(j)}=\Psi\pmb{g}_j)\}$.
\State Compute $\mathrm{res}^{\mathrm{var}}(\lambda_j,g_{(j)})$ for all $j$ (see~\eqref{eq:abs_res}).
\end{algorithmic} \textbf{Output:} Eigenpairs $\{(\lambda_j,\pmb{g}_j)\}$ and variance residuals $\{\mathrm{res}^{\mathrm{var}}(\lambda_j,g_{(j)})\}$.
\caption{: \textbf{Eigenpairs and residuals.}}\label{alg:stoch_resDMD1}
\end{algorithm}

\begin{algorithm}[t]
\textbf{Input:} Snapshot data $\{\pmb{x}^{(m)}\}_{m=1}^{M},\{\pmb{y}^{(m,1)},\pmb{y}^{(m,2)}\}_{m=1}^{M}$ (batched), quadrature weights $\{w_m\}_{m=1}^{M}$, dictionary of observables $\{\psi_j\}_{j=1}^{N}$.\\
\vspace{-4mm}
\begin{algorithmic}[1]
\State Compute
\begin{align*}
\tilde{G}&=\Psi_X^*W\Psi_X,\\
\tilde{A}&=\left[\Psi_X^*W\Psi_Y^{(1)}+\Psi_X^*W\Psi_Y^{(2)}\right]/2,\\
\tilde{L}&=\left[{\Psi_Y^{(1)}}^*W\Psi_Y^{(1)}+{\Psi_Y^{(2)}}^*W\Psi_Y^{(2)}\right]/2,\\
\tilde{H}&=\left[{\Psi_Y^{(1)}}^*W\Psi_Y^{(2)}+{\Psi_Y^{(2)}}^*W\Psi_Y^{(1)}\right]/2,
\end{align*}
where $\Psi_X$ and $\Psi_Y^{(i)}$ are given in~\eqref{PSI_defs} and the superscript for $\Psi_Y$ corresponds to each batch of snapshot data. 
\State Solve $\tilde{A}\pmb{g}=\lambda \tilde{G}\pmb{g}$ for eigenpairs $\{(\lambda_j,g_{(j)}=\Psi\pmb{g}_j)\}$.
\State Compute $\mathrm{res}^{\mathrm{var}}(\lambda_j,g_{(j)})$ and $\mathrm{res}(\lambda_j,g_{(j)})$ for all $j$ (see~\eqref{eq:abs_res} and~\eqref{eq:abs_res2}).
\end{algorithmic} \textbf{Output:} Eigenpairs $\{(\lambda_j,\pmb{g}_j)\}$ and residuals $\{\mathrm{res}^{\mathrm{var}}(\lambda_j,g_{(j)}),\mathrm{res}(\lambda_j,g_{(j)})\}$.
\caption{: \textbf{Eigenpairs and residuals (batched data).}}\label{alg:stoch_resDMD2}
\end{algorithm}

\begin{algorithm}[t]
\textbf{Input:} Snapshot data $\{\pmb{x}^{(m)}\}_{m=1}^{M},\{\pmb{y}^{(m,1)},\pmb{y}^{(m,2)}\}_{m=1}^{M}$ (batched), quadrature weights $\{w_m\}_{m=1}^{M}$, dictionary of observables $\{\psi_j\}_{j=1}^{N}$, an accuracy goal $\epsilon>0$, and a grid $z_1,\ldots,z_k\in\mathbb{C}$ (e.g., see \eqref{eq:grid_def}).\\
\vspace{-4mm}
\begin{algorithmic}[1]
\State Compute
\begin{align*}
\tilde{G}&=\Psi_X^*W\Psi_X,\\
\tilde{A}&=\left[\Psi_X^*W\Psi_Y^{(1)}+\Psi_X^*W\Psi_Y^{(2)}\right]/2,\\
\tilde{L}&=\left[{\Psi_Y^{(1)}}^*W\Psi_Y^{(1)}+{\Psi_Y^{(2)}}^*W\Psi_Y^{(2)}\right]/2,\\
\tilde{H}&=\left[{\Psi_Y^{(1)}}^*W\Psi_Y^{(2)}+{\Psi_Y^{(2)}}^*W\Psi_Y^{(1)}\right]/2,
\end{align*}
where $\Psi_X$ and $\Psi_Y^{(i)}$ are given in~\eqref{PSI_defs} and the superscript for $\Psi_Y$ corresponds to each batch of snapshot data.
\State For each $z_j$, compute $r_j = \min_{\pmb{g}\in\mathbb{C}^{N}} \mathrm{res}(z_j,\Psi\pmb{g})$ (see~\eqref{eq:abs_res2}) and the corresponding singular vectors $\pmb{g}_j$. This step is a generalized SVD problem.
\end{algorithmic} \textbf{Output:} $\{z_j: r_j<\epsilon\}$, an estimate of $\mathrm{Sp}_\epsilon(\mathcal{K}_{(1)})$, and pseudoeigenfunctions $\{\pmb{g}_j: r_j<\epsilon\}$.
\caption{: \textbf{Pseudospectra (batched data).}}\label{alg:res_pseudospec1}
\end{algorithm}

\begin{algorithm}[t]
\textbf{Input:} Snapshot data $\{\pmb{x}^{(m)}\}_{m=1}^{M},\{\pmb{y}^{(m)}\}_{m=1}^{M}$ ($\pmb{y}^{(m)}=F(\pmb{x}^{(m)},\tau_m)$), quadrature weights $\{w_m\}_{m=1}^{M}$, dictionary of observables $\{\psi_j\}_{j=1}^{N}$, an accuracy goal $\epsilon>0$, and a grid $z_1,\ldots,z_k\in\mathbb{C}$ (e.g., see \eqref{eq:grid_def}).\\
\vspace{-4mm}
\begin{algorithmic}[1]
\State Compute
\begin{align*}
\tilde{G}&=\Psi_X^*W\Psi_X,\\
\tilde{A}&=\Psi_X^*W\Psi_Y,\\
\tilde{L}&=\Psi_Y^*W\Psi_Y,
\end{align*}
where $\Psi_X$ and $\Psi_Y$ are given in~\eqref{PSI_defs}.
\State For each $z_j$, compute $r_j = \min_{\pmb{g}\in\mathbb{C}^{N}} \mathrm{res}^{\mathrm{var}}(z_j,\Psi\pmb{g})$ (see~\eqref{eq:abs_res}) and the corresponding singular vectors $\pmb{g}_j$. This step is a generalized SVD problem.
\end{algorithmic} \textbf{Output:} $\{z_j: r_j<\epsilon\}$, an estimate of $\mathrm{Sp}_\epsilon^\mathrm{var}(\mathcal{K}_{(1)})$, and variance-pseudoeigenfunctions $\{\pmb{g}_j: r_j<\epsilon\}$.
\caption{: \textbf{Variance-pseudospectra.}}\label{alg:res_pseudospec2}
\end{algorithm}

\section{Theoretical guarantees}\label{sec:theory}

We now prove the correctness of the algorithms mentioned above. Specifically, through a series of theorems, we demonstrate that the computations of $\tilde{A},\tilde{G},\tilde{L}$, and $\tilde{H}$ are accurate and that the spectral estimates can be trusted. To achieve this, we divide the section into three subsections, each focusing on demonstrating the accuracy of the spectrum, the predictive power, and the matrices, respectively. The universal assumptions made in this section are as follows:
\begin{itemize}
	\item $\mathcal{K}_{(1)}$ is bounded.
	\item $\{\psi_j\}_{j=1}^N$ are linearly independent for any finite $N$.
	\item $V_N\subset V_{N+1}$ and the union, $\cup_N V_N$, is dense in $L^2(\Omega,\omega)$.
\end{itemize}
The algorithms and proofs can be readily adapted for an unbounded $\mathcal{K}_{(1)}$. The latter two assumptions can also be relaxed with minor modifications.

\subsection{Accuracy in finding spectral quantities}\label{sec:spectrum}
In this subsection, we prove the convergence of our algorithms. We have already discussed the convergence of residuals in Algorithms \ref{alg:stoch_resDMD1} and \ref{alg:stoch_resDMD2}, under the assumption of convergence of the finite matrices $\tilde{G},\tilde{A},\tilde{L}$, and $\tilde{H}$ in the large data limit. Hence, we focus on Algorithm \ref{alg:res_pseudospec2}. We first define the functions
$$
f_{M,N}(\lambda)=\min_{\pmb{g}\in\mathbb{C}^{N}} \mathrm{res}^{\mathrm{var}}(\lambda,\Psi \pmb{g}),
$$
and note that $r_j=f_{M,N}(z_j)$ in Algorithm~\ref{alg:res_pseudospec2}. Our first lemma describes the limit of these functions as $M\rightarrow\infty$ and $N\rightarrow\infty$.

\begin{lemma}
\label{techno2}
Suppose that
$$
\lim_{M\rightarrow\infty} \tilde{G}=G,\quad \lim_{M\rightarrow\infty} \tilde{A}=A,\quad \lim_{M\rightarrow\infty} \tilde{L}=L,
$$
then $f_N(\lambda)=\lim_{M\rightarrow\infty}f_{M,N}(\lambda)$ exists. Moreover, $f_N$ is a nonincreasing function of $N$ and converges to $\sigma_{\mathrm{inf}}^{\mathrm{var}}$ from above and uniformly on compact subsets of $\mathbb{C}$ as a function of the spectral parameter $\lambda$.
\end{lemma}

\begin{proof}
The limit $f_N(\lambda)=\lim_{M\rightarrow\infty}f_{M,N}(\lambda)$ follows trivially from the convergence of matrices. Moreover, we have
\begin{align*}
f_N(\lambda)&=\min_{\pmb{g}\in\mathbb{C}^{N}}\sqrt{\frac{\pmb{g}^*(L-\lambda A^*-\overline{\lambda}A+|\lambda|^2G)\pmb{g}}{\pmb{g}^*G\pmb{g}}}\\
&=\inf\left\{\sqrt{\mathbb{E}_{\tau}\left[\|g\circ F_\tau-\lambda g\|^2\right]}:g\in V_N,\|g\|=1\right\}.
\end{align*}
Since ${V}_{N}\subset {V}_{N+1}$, $f_N(\lambda)$ is nonincreasing in $N$. By definition, we also have
$$
f_N(\lambda)\geq \sigma_{\mathrm{inf}}^{\mathrm{var}}(\lambda).
$$
Let $\delta>0$ and choose $g\in L^2(\Omega,\omega)$ such that $\|g\|=1$ and
$$
\sqrt{\mathbb{E}_{\tau}\left[\|g\circ F_\tau-\lambda g\|^2\right]}\leq \sigma_{\mathrm{inf}}^{\mathrm{var}}(\lambda)+\delta.
$$
Since $\cup_N V_N$ is dense in $L^2(\Omega,\omega)$, there exists some $n$ and $g_{n}\in{V}_{n}$ such that $\|g_n\|=1$ and
$$
\sqrt{\mathbb{E}_{\tau}\left[\|g_n\circ F_\tau-\lambda g_n\|^2\right]}\leq  \sqrt{\mathbb{E}_{\tau}\left[\|g\circ F_\tau-\lambda g\|^2\right]}+\delta.
$$
It follows that $f_n(\lambda)\leq \sigma_{\mathrm{inf}}^{\mathrm{var}}(\lambda)+2\delta $. Since this holds for any $\delta>0$, $\lim_{N\rightarrow\infty}f_N(\lambda)=\sigma_{\mathrm{inf}}^{\mathrm{var}}(\lambda)$. Since $\sigma_{\mathrm{inf}}^{\mathrm{var}}(\lambda)$ is continuous in $\lambda$, $f_N$ converges uniformly down to $\sigma_{\mathrm{inf}}^{\mathrm{var}}$ on compact subsets of $\mathbb{C}$ by Dini's theorem.\qed
\end{proof}

Let $\{\mathrm{Grid}(N)=\{z_{1,N},z_{2,N},\ldots,z_{k(N),N}\}\}$ be a sequence of grids, each finite, such that for any $\lambda\in\mathbb{C}$,
$$
\lim_{N\rightarrow\infty}\mathrm{dist}(\lambda,\mathrm{Grid}(N))=0.
$$
For example, we could take
\begin{equation}
\label{eq:grid_def}
\mathrm{Grid}(N)=\frac{1}{N}\left[\mathbb{Z}+i\mathbb{Z}\right]\cap \{z\in\mathbb{C}:|z|\leq N\}.
\end{equation}
In practice, one considers a grid of points over the region of interest in the complex plane. Lemma \ref{techno2} tells us that to study Algorithm \ref{alg:res_pseudospec2} in the large data limit, we must analyze
$$
\Gamma^\epsilon_{N}(\mathcal{K}_{(1)})=\left\{\lambda\in\mathrm{Grid}(N):f_N(\lambda)<\epsilon\right\}.
$$
To make the convergence of Algorithm \ref{alg:res_pseudospec2} precise, we use the Attouch--Wets metric defined by~\cite{beer1993topologies}:
$$
d_{\mathrm{AW}}(C_1,C_2)=\sum_{n=1}^{\infty} 2^{-n}\min\big\{{1,\underset{\left|x\right|\leq n}{\sup}\left|\mathrm{dist}(x,C_1)\!-\!\mathrm{dist}(x,C_2)\right|}\big\},
$$
where $C_1,C_2$ are closed nonempty subsets of $\mathbb{C}$. This metric corresponds to local uniform converge on compact subsets of $\mathbb{C}$. For any closed nonempty sets $C$ and $C_n$, $d_{\mathrm{AW}}(C_n,C)\rightarrow{0}$ if and only if for any $\delta>0$ and $B_m(0)$ (closed ball of radius $m\in\mathbb{N}$ about $0$), there exists $N$ such that if $n>N$ then $C_n\cap B_m(0)\subset{C+B_{\delta}(0)}$ and $C\cap B_m(0)\subset{C_n+B_{\delta}(0)}$. The following theorem contains our convergence result.

\begin{theorem}[Convergence to variance-pseudospectrum]
\label{half_pseudospectrum}
Let $\epsilon>0$. Then, $\Gamma^\epsilon_{N}(\mathcal{K}_{(1)})\subset\mathrm{Sp}_{\epsilon}^{\mathrm{var}}(\mathcal{K}_{(1)})$ and
$$
\lim_{N\rightarrow\infty}d_{\mathrm{AW}}\left(\Gamma^\epsilon_{N}(\mathcal{K}_{(1)}),\mathrm{Sp}_{\epsilon}^{\mathrm{var}}(\mathcal{K}_{(1)})\right)=0.
$$
\end{theorem}

\begin{proof}
Lemma \ref{techno2} shows that $\Gamma^\epsilon_{N}(\mathcal{K}_{(1)})\subset\mathrm{Sp}_{\epsilon}^{\mathrm{var}}(\mathcal{K}_{(1)})$. To prove convergence, we use the characterization of the Attouch--Wets topology. Suppose that $m$ is large such that $B_m(0)\cap\mathrm{Sp}_{\epsilon}^{\mathrm{var}}(\mathcal{K}_{(1)})\neq\emptyset$. Since $\Gamma^\epsilon_{N}(\mathcal{K}_{(1)})\subset\mathrm{Sp}_{\epsilon}^{\mathrm{var}}(\mathcal{K}_{(1)})$, we clearly have $\Gamma_{N}^{\epsilon}(\mathcal{K}_{(1)})\cap B_m(0)\subset\mathrm{Sp}_{\epsilon}^{\mathrm{var}}(\mathcal{K}_{(1)})$. Hence, we must show that given $\delta>0$, there exists $n_0$ such that if $N>n_0$ then $\mathrm{Sp}_{\epsilon}^{\mathrm{var}}(\mathcal{K}_{(1)})\cap B_m(0)\subset{\Gamma_{N}^{\epsilon}(\mathcal{K}_{(1)})+B_{\delta}(0)}$. Suppose for a contradiction that this statement is false. Then, there exists $\delta>0$, $\lambda_{n_j}\in\mathrm{Sp}_{\epsilon}^{\mathrm{var}}(\mathcal{K}_{(1)})\cap B_m(0)$, and $n_j\rightarrow\infty$ such that
$$
\mathrm{dist}(\lambda_{n_j},\Gamma_{n_j}^{\epsilon}(\mathcal{K}_{(1)}))\geq\delta.
$$
Without loss of generality, we can assume that $\lambda_{n_j}\rightarrow \lambda\in\mathrm{Sp}_{\epsilon}^{\mathrm{var}}(\mathcal{K}_{(1)})\cap B_m(0)$. There exists some $z$ with $\sigma_{\mathrm{inf}}^{\mathrm{var}}(z)<\epsilon$ and $\left|\lambda-z\right|\leq \delta/2$. Let $z_{n_j}\in\mathrm{Grid}(n_j)$ such that $|z-z_{n_j}|\leq \mathrm{dist}(z,\mathrm{Grid}(n_j))+{n_j}^{-1}.$ Since $\sigma_{\mathrm{inf}}^{\mathrm{var}}$ is continuous and $f_N$ converges locally uniformly to $\sigma_{\mathrm{inf}}^{\mathrm{var}}$, we must have $f_{n_j}(z_{n_j})<\epsilon$ for large $n_j$ so that $z_{n_j}\in\Gamma_{n_j}^{\epsilon}(\mathcal{K}_{(1)})$. But
$
\left|z_{n_j}-\lambda\right|\leq \left|z-\lambda\right|+\left|z_{n_j}-z\right|\leq \delta/2 + |z-z_{n_j}|,
$
which is smaller than $\delta$ for large $n_j$, and we reach the desired contradiction.\qed
\end{proof}

\subsection{Error bounds for iterations}
\label{sec:forecast_bound}

We now aim to bound the difference between $\tilde{K}^n$ and $\mathcal{K}^n$, a step crucial for measuring the accuracy of our approximation of the mean trajectories in $L^2(\Omega,\omega)$. This effort, in conjunction with the Chernoff-like bound presented in \eqref{chernoff_bound}, enables us to compute the statistical properties of the trajectories and their forecasts. Our approach to establishing these bounds is twofold. First, we consider the difference between $\tilde{K}^n$ and $\mathcal{K}^n$, taking into account both the estimation errors and the errors intrinsic to the subspace. Subsequently, we establish concentration bounds for the estimation errors of $\tilde{G}$, $\tilde{A}$, and $\tilde{L}$.

\begin{theorem}[Error bound for forecasts]
\label{thm:forecast_bound}
Define the quantities
\begin{align*}
I_G&=G^{\frac{1}{2}}\tilde{G}^{-\frac{1}{2}},\\
\Delta_G&=\|I_G\|\|(I-I_G^{-1})\|+\|(I-I_G)\|,\\
\Delta_A&=\|\mathcal{K}\|(1+\|I_G\|)\|I_G-I\|+\|I_G\|^2\|G^{-\frac{1}{2}}(A-\tilde{A})G^{-\frac{1}{2}}\|.
\end{align*}
Let $g=\sum_{j=1}^N\pmb{g}_j\psi_j\in V_N$ and suppose that
$$
\|\mathcal{K}^n_{(1)}g-\mathcal{P}_{V_N}^*(\mathcal{P}_{V_N}\mathcal{K}_{(1)}\mathcal{P}_{V_N}^*)^ng\|\leq \delta_n(g)\|g\|.
$$
Then
$$
\|\Psi\tilde{K}^n\pmb{g}-\mathcal{K}^n_{(1)}g\|\leq C_n\|g\|,
$$
where
$$
C_n=\left[\frac{\|\mathcal{K}\|^n-\Delta_A^n}{\|\mathcal{K}\|-\Delta_A}\Delta_A(\Delta_G+1)+\|\mathcal{K}\|^n\Delta_G+\delta_n(g)\right]\,.
$$
\end{theorem}
\begin{proof}
We introduce the two matrices
$$
T=G^{-1/2}AG^{-1/2},\quad \tilde{T}=\tilde{G}^{-1/2}\tilde{A}\tilde{G}^{-1/2}.
$$
Note that
\begin{align*}
\|T\|=\sup_{x\in\mathbb{C}^N}\frac{\|TG^{1/2}x\|}{\|G^{1/2}x\|}&=\sup_{x\in\mathbb{C}^N}\frac{\|G^{1/2}Kx\|}{\|G^{1/2}x\|}\\
&=\|\mathcal{P}_{V_{N}}\mathcal{K}\mathcal{P}_{V_{N}}^*\|\leq \|\mathcal{K}\|.
\end{align*}
We can re-write $\tilde{T}$ as
\begin{align*}
\tilde{T}&=I_G^*G^{-1/2}\tilde A G^{-1/2}I_G\\
&=I_G^*TI_G+I_G^*G^{-1/2}(\tilde A-A)G^{-1/2}I_G\\
&=T+(I_G-I)^*TI_G+T(I_G-I)\\
&\quad\quad+I_G^*G^{-1/2}(\tilde A-A)G^{-1/2}I_G.
\end{align*}
It follows that
\begin{align*}
\|T-\tilde{T}\|&\leq \|\mathcal{K}\|(1+\|I_G\|)\|I_G-I\|\\
&\quad\quad+\|I_G\|^2\|G^{-1/2}(A-\tilde{A})G^{-1/2}\|\\
&=\Delta_A.
\end{align*}
We have that
$$
T^n-\tilde{T}^n=T(T^{n-1}-\tilde{T}^{n-1})+({T}-\tilde{T})\tilde{T}^{n-1}.
$$
A simple proof by induction now shows that
\begin{align*}
\|T^n-\tilde{T}^n\|&\leq \|{T}-\tilde{T}\|\sum_{j=0}^{n-1}\|T\|^j\|\tilde{T}\|^{n-1-j}\\
&\leq \Delta_A\sum_{j=0}^{n-1}\|\mathcal{K}\|^{j}(\|\mathcal{K}\|+\Delta_A)^{n-1-j}\\
&= \Delta_A\frac{\|\mathcal{K}\|^n-\Delta_A^n}{\|\mathcal{K}\|-\Delta_A}.
\end{align*}
We wish to bound the quantity
\begin{align*}
&\|\Psi K^n\pmb{g}-\Psi\tilde{K}^n\pmb{g}\|=\|{T}^n{G}^{1/2}\pmb{g}-I_G\tilde{T}^n\tilde{G}^{1/2}\pmb{g}\|\\
&\quad\quad\quad\leq \|{T}^n-\tilde{T}^n\|\|g\|+\|\tilde{T}^n{G}^{1/2}\pmb{g}-I_G\tilde{T}^n\tilde{G}^{1/2}\pmb{g}\|.
\end{align*}
We can express the final term on the right-hand side as
\begin{align*}
\tilde{T}^n{G}^{1/2}\pmb{g}-I_G{\tilde{T}}^n\tilde{G}^{1/2}\pmb{g}&=I_G\tilde{T}^n(I-I_G^{-1}){G}^{1/2}\pmb{g}\\
&\quad\quad+(I-I_G){\tilde{T}}^n{G}^{1/2}\pmb{g}.
\end{align*}
It follows that
\begin{align*}
\|\tilde{T}^n{G}^{1/2}\pmb{g}-I_G\tilde{T}^n\tilde{G}^{1/2}\pmb{g}\|&\leq \|\tilde{T}^n\|\|{G}^{1/2}\pmb{g}\|\Delta_G\\
&\leq \left(\|\mathcal{K}\|^n+\|{T}^n-\tilde{T}^n\|\right)\Delta_G\|g\|
\end{align*}
and hence that
\begin{align*}
&\|\Psi K^n\pmb{g}-\Psi\tilde{K}^n\pmb{g}\|\leq \left[\|{T}^n-\tilde{T}^n\|(\Delta_G+1)+\|\mathcal{K}\|^n\Delta_G\right]\|g\|\\
&\quad\quad\quad\quad\leq \left[\frac{\|\mathcal{K}\|^n-\Delta_A^n}{\|\mathcal{K}\|-\Delta_A}\Delta_A(\Delta_G+1)+\|\mathcal{K}\|^n\Delta_G\right]\|g\|.
\end{align*}
The theorem now follows from the triangle inequality.\qed
\end{proof}

This theorem explicitly tells us how much to trust the prediction using the computed Koopman matrix, compared with the true Koopman operator. The quantities $\Delta_G$ and $\Delta_A$ represent errors due to estimation or quadrature. They are both expected to be small. The quantity $\delta_n(g)$ is an intrinsic invariant subspace error that depends on the dictionary and observable $g$. To approximate $\delta_n(g)$, note that
$$
\mathcal{K}^{n}[g]{-}\Psi K^n\pmb{g}=\sum_{j=1}^n\mathcal{K}^{n-j}[\mathcal{K}[\Psi K^{j-1}\pmb{g}]{-}\Psi K^j \pmb{g}]
$$
and hence
\begin{equation}
\label{delta_comp1}
\|\mathcal{K}^{n}[g]{-}\Psi K^n\pmb{g}\|{\leq}\!\sum_{j=1}^n\!\|\!\mathcal{K}\!\|^{n{-}j}\|\!\mathcal{K}\![\Psi K^{j{-}1}\pmb{g}]{-}\Psi K^j \pmb{g}\!\|.
\end{equation}
To bound the term on the right-hand side, we can use the matrix $H$ in \eqref{eqn:new_H_matrix} and the fact that
\begin{equation}
\label{delta_comp2}
\|\mathcal{K}\Psi \pmb{v}{-}\Psi Kv\|=\sqrt{\pmb{v}^*H\pmb{v}{-}2\mathrm{Re}(\pmb{v}^*K^*A\pmb{v}){+}\pmb{v}^*K^*GK\pmb{v}}
\end{equation}
for any $\pmb{v}\in\mathbb{C}^N$.

\subsection{Estimation error for computation of $A$, $G$, and $L$}\label{sec:matrix}

To effectively estimate $\mathcal{K}_{(1)}g$ and ${\rm Sp}_\epsilon^{\rm var}(\mathcal{K}_{(1)})$ in practical applications, it is imperative to have reliable approximations of $A$, $G$, and $L$. We provide a justification for our ability to construct such approximations from trajectory data with high probability, employing concentration bounds. The subsequent result delineates the requisite number of samples and basis functions needed to achieve a desired level of accuracy with high probability. To ensure this level of accuracy, several reasonable assumptions about the stochastic dynamical system are necessary.

\begin{assumption}\label{ass:error}
We suppose that $\pmb{x}^{(m)}$ in the snapshot data are sampled at random according to $\omega$, independent of $\tau$, and for simplicity, assume that $\omega$ is a probability measure.\footnote{Similar types of bounds to Theorem \ref{thm:prob_concentration} can be derived for ergodic sampling and high-order quadrature sampling.} We assume that $\tau:\Omega_s\rightarrow\mathcal{H}$ for some Hilbert space $\mathcal{H}$ and let $\kappa=(\pmb{x},\tau)$. In this section, $\mathbb{E}$ and $\mathbb{P}$ are with respect to the joint distribution of $\kappa$. We assume that
\begin{itemize}
\item The random variable $\kappa$ is sub-Gaussian, meaning that there exists some $a>0$ such that
$$
\mathbb{E}\left[e^{\|\kappa-\mathbb{E}(\kappa)\|^2/a^2}\right]<\infty.
$$
This allows us to define the following finite quantity:
$$
\Upsilon=\inf\left\{s>0:e^{\frac{\mathbb{E}[\|\kappa-\mathbb{E}(\kappa)\|^2]}{s^2}}\mathbb{E}\left[e^{\frac{1}{s^2}\|\kappa-\mathbb{E}(\kappa)\|^2}\right]\leq 2\right\}.
$$
\item The dictionary functions are uniformly bounded and satisfy the following Lipschitz condition:
$$
|\psi_k(\pmb{x})-\psi_k(\pmb{x}')|\leq c_k\|\pmb{x}-\pmb{x}'\|.
$$
\item The function $F$ is Lipschitz with
$$
\|F(\kappa)-F(\kappa')\|\leq c\|\kappa-\kappa'\|.
$$
\end{itemize}
\end{assumption}

With these assumptions, we can show that our approximations of $A$, $G$, and $L$ are good with high probability.  

\begin{theorem}[Concentration bound on estimation errors]
\label{thm:prob_concentration}
Under Assumption~\ref{ass:error} we have, for any $t>0$,
\begin{align*}
&\mathbb{P}\left(\|\tilde{A}{-} A\|_{\mathrm{Fr}}< t\right){\geq}1{-}\exp\left(\!2\log(2N){-}\frac{Mt^2}{24\Upsilon^2(c^2{+}1)\alpha^2\beta^2}\right)\\
&\mathbb{P}\left(\|\tilde{G}{-} G\|_{\mathrm{Fr}}< t\right){\geq}1{-}\exp\left(\!2\log(2N){-}\frac{Mt^2}{48\Upsilon^2\alpha^2\beta^2}\right)\\
&\mathbb{P}\left(\|\tilde{L}{-} L\|_{\mathrm{Fr}}< t\right){\geq}1{-}\exp\left(\!2\log(2N){-}\frac{Mt^2}{48\Upsilon^2c^2\alpha^2\beta^2}\right),
\end{align*}
where $\|\cdot\|_{\mathrm{Fr}}$ denotes the Frobenius norm, and $\alpha$ and $\beta$ are given by 
$$
\alpha=\sqrt{\sum_{k=1}^Nc_k^2},\quad \beta =\sqrt{\sum_{k=1}^N\|\psi_k\|_{L^\infty}^2}.
$$
\end{theorem}
\begin{proof}
We first argue for $\|\tilde{A}-A\|_{\mathrm{Fr}}$. Fix $j,k\in\{1,\ldots,N\}$ and define the random variable
$$
X=\psi_k(F(\pmb{x},\tau))\overline{\psi_j(\pmb{x})}.
$$
Then
$$
\left|X(\kappa)-X(\kappa')\right|\leq (c_kc\|\psi_j\|_{L^\infty}+c_j\|\psi_k\|_{L^\infty})\|\kappa-\kappa'\|.
$$
Let $c_{j,k}=c_kc\|\psi_j\|_{L^\infty}+c_j\|\psi_k\|_{L^\infty}$. The above Lipschitz bound for $X$ implies that
\begin{align*}
\left|\mathbb{E}[X]-X(\kappa')\right|&\leq c_{j,k}\int_{\Omega\times\Omega_s}\|\kappa-\kappa'\| \,\mathrm{d} \mathbb{P}(\kappa)\\
&\leq c_{j,k}\sqrt{\|\kappa-\mathbb{E}(\kappa)\|^2+\mathbb{E}(\|\kappa-\mathbb{E}(\kappa)\|^2)},
\end{align*}
where we have used H\"older's inequality to derive the last line. It follows that
$$
\mathbb{E}\left[\exp\left(\frac{\left|\mathbb{E}[X]-X\right|^2}{\Upsilon^2c_{j,k}^2}\right)\right]\leq 2.
$$
Let $Y=\mathrm{Re}\left(\mathbb{E}\left[X\right]-X\right)$ and $\lambda\geq 0$. Since $\mathbb{E}[Y]=0$, we have
$$
\mathbb{E}\left[\exp\left(\lambda Y\right)\right]=1+\sum_{l=2}^\infty\frac{\lambda^l\mathbb{E}[Y^l]}{l!}\leq 1+\frac{\lambda^2}{2}\mathbb{E}\left[Y^2\exp(\lambda|Y|)\right].
$$
For any $b>0$, we have $\lambda|Y|\leq\lambda^2/(2b)+b|Y|^2/2$. We also have $bY^2\leq \exp(bY^2/2)$. It follows that
$$
\mathbb{E}\left[\exp\left(\lambda Y\right)\right]\leq 1+\frac{\lambda^2}{2b}e^{\lambda^2/(2b)}\mathbb{E}\left[\exp(bY^2)\right].
$$
We select $b=1/(\Upsilon^2c_{j,k}^2)$ and use the fact that $\mathbb{E}\left[\exp(bY^2)\right]\leq \mathbb{E}\left[\exp(b|\mathbb{E}[X]-X|^2)\right]\leq 2$ to obtain
$$
\mathbb{E}\left[\exp\left(\lambda Y\right)\right]
\leq 1+\frac{\lambda^2}{b}e^{\frac{\lambda^2}{2b}}
\leq\left(1+\frac{\lambda^2}{b}\right)e^{\frac{\lambda^2}{2b}}
\leq e^{\frac{3\lambda^2}{2b}}.
$$
Now let $\{Y^{(m)}\}_{m=1}^{M}$ independent copies of $Y$, then
\begin{align*}
&\mathbb{P}\left(\frac{1}{M}\sum_{m=1}^{M}Y^{(m)}\geq t\right)=\mathbb{P}\left(\exp(\lambda\sum_{m=1}^{M}Y^{(m)})\geq \exp(\lambda Mt) \right)\\
&\leq e^{-\lambda Mt}\mathbb{E}\left[\exp\left(\lambda\sum_{m=1}^{M}Y^{(m)}\right)\right]= e^{-\lambda Mt}\prod_{m=1}^{M}\mathbb{E}\left[\exp\left(\lambda Y\right)\right]\\
&\leq \exp\left(3M\lambda^2/(2b)-\lambda M t\right),
\end{align*}
where we use Markov's inequality in the first inequality. Minimizing over $\lambda$, we obtain
$$
\mathbb{P}\left(\frac{1}{M}\sum_{m=1}^{M}Y^{(m)}\geq t\right)\leq \exp\left(-Mbt^2/6\right).
$$
We can argue in the same manner for $-Y$ and deduce that
$$
\mathbb{P}\left(\frac{1}{M}\left|\sum_{m=1}^{M}Y^{(m)}\right|\geq t\right)\leq 2\exp\left(-Mbt^2/6\right).
$$
Similarly, we can argue for the imaginary part of $\mathbb{E}[X]-X$.

We now allow $j,k$ to vary and let $X_{j,k}=\psi_k(F(\pmb{x},\tau))\overline{\psi_j(\pmb{x})}$. For $t>0$, consider the events
\begin{align*}
S_{j,k,1}&:\frac{1}{M}\left|\sum_{m=1}^{M}\mathrm{Re}\left(\mathbb{E}[X_{j,k}]-X_{j,k}(\kappa_m)\right)\right|< \frac{t\Upsilon c_{j,k}}{\sqrt{2\Upsilon^2 \sum_{l,p=1}^Nc_{l,p}^2}},\\
S_{j,k,2}&:\frac{1}{M}\left|\sum_{m=1}^{M}\mathrm{Im}\left(\mathbb{E}[X_{j,k}]-X_{j,k}(\kappa_m)\right)\right|< \frac{t\Upsilon c_{j,k}}{\sqrt{2\Upsilon^2 \sum_{l,p=1}^Nc_{l,p}^2}}.
\end{align*}
Then
\begin{align*}
\mathbb{P}(\cap_{j,k,i}S_{j,k,i})&\geq 1 - \sum_{j,k=1}^N (\mathbb{P}(S_{j,k,1}^c)+\mathbb{P}(S_{j,k,2}^c))\\
&\geq 1-4N^2\exp\left(-\frac{Mt^2}{12\Upsilon^2\sum_{l,p=1}^Nc_{l,p}^2}\right).
\end{align*}
Moreover, the AM-GM inequality implies that
$$
c_{l,p}^2\leq 2c^2c_k^2\|\psi_j\|_{L^\infty}^2+2c_j^2\|\psi_k\|_{L^\infty}^2
$$
and hence
$$
\sum_{l,p=1}^Nc_{l,p}^2\leq 2(c^2+1)\alpha^2\beta^2.
$$
It follows that
$$
\mathbb{P}(\cap_{j,k,i}S_{j,k,i})\geq 1-\exp\left(2\log(2N)-\frac{Mt^2}{24\Upsilon^2(c^2+1)\alpha^2\beta^2}\right).
$$
If $\cap_{j,k,i}S_{j,k,i}$, then $\|\tilde{A}-A\|_{\mathrm{Fr}}< t$. We can argue in the same manner, without the function $F$, to deduce that
$$
\mathbb{P}(\|\tilde{G}-G\|_{\mathrm{Fr}}< t)\geq 1-\exp\left(2\log(2N)-\frac{Mt^2}{48\Upsilon^2\alpha^2\beta^2}\right).
$$
Finally, for the matrix $L$ and its estimate $\tilde{L}$, we derive similar concentration bounds for $\psi_k(F(\pmb{x},\tau))\overline{\psi_j(F(\pmb{x},\tau))}$  to see that
$$
\mathbb{P}(\|\tilde{L}-L\|_{\mathrm{Fr}}< t)\geq 1-\exp\left(2\log(2N)-\frac{Mt^2}{48\Upsilon^2c^2\alpha^2\beta^2}\right).
$$
The statement of the theorem now follows.
\qed
\end{proof}

This theorem explicitly spells out the number of basis functions and samples required to approximate the three matrices appearing in Theorem \ref{thm:forecast_bound}. Roughly speaking, if we set
$$
\exp\left(\!2\log(2N)\!-\!{Mt^2}\right)\sim N^2\exp\left(-Mt^2\right)\leq\delta\,,
$$
then
$$
M\sim |\ln{\delta}-2\ln{N}|/{t^2}.
$$
For any fixed tolerance $t$, the confidence exponentially tightens up when $M$, the number of samples, increases. The idea is similar to other concentration inequality type bounds: if one samples from the same distribution many times, the sample mean becomes closer and closer to the true mean, and this bound gives the confidence interval for the tail bound. On the other hand, when $N$ increases, more entries in the matrices need to be approximated, so it brings a logarithmically negative effect. More samples are needed to balance out the increase of $N$.

\section{Examples}\label{sec:examples}

We now present three examples. The first two are based on numerically sampled trajectory data, while the final example utilizes collected experimental data.

\subsection{Arnold's circle map}
\label{sec:cat_map}

For our first example, we revisit the circle map discussed in Example \ref{example:circle_map}, setting 
$c=1/5$, $\rho$ as the uniform distribution on $[0,1]$, and defining
$$
f(\pmb{x})=\frac{1}{4\pi}\sin(2\pi \pmb{x}).
$$
Our dictionary consists of Fourier modes $\{\exp(ij\pmb{x}):j=-n,\ldots,n\}$ with $n=20$ (yielding $N=41$),  and we use batched trajectory data with $M_1=100$ equally spaced $\{\pmb{x}^{(j)}\}$, and $M_2=2\times 10^4$. Figure \ref{fig:circle_con} illustrates the convergence of the matrices $\tilde{A},\tilde{L}$, and $\tilde{H}$. We do not display the convergence of $\tilde{G}$ as its error was on the order of machine precision, a result of the exponential convergence achieved by the trapezoidal quadrature rule across different batches. Figure \ref{fig:circle_res} shows the residuals computed using Algorithm \ref{alg:stoch_resDMD2}. The quantity $\mathrm{res}^{\mathrm{var}}(\lambda,g)$ deviates from \eqref{circle_var_no_f} (the formula for $f=0$), particularly when $|\lambda|$ is small. As $n$ increases, the residuals $\mathrm{res}(\lambda,g)$ converge to zero, indicating more accurate computation of the spectral content of $\mathcal{K}_{(1)}$. However, the residuals $\mathrm{res}^{\mathrm{var}}(\lambda,g)$ converge to finite positive values, except for the trivial eigenvalue $1$, which satisfies $\lim_{M\rightarrow\infty}\mathrm{res}^{\mathrm{var}}(\lambda,g)=0$.

To underscore the significance of variance in our analysis, Figure \ref{fig:circle_cov} displays the absolute value of the matrix $\tilde{L}-\tilde{H}$, which approximates the covariance matrix defined in \eqref{cov_mat_def}. Notably, the covariance disappears for the constant function $\exp(ij\pmb{x})$ with $j=0$, and the matrix is diagonally dominated. Figure \ref{fig:circle_pseudospec} presents the results obtained from applying Algorithms \ref{alg:res_pseudospec1} and \ref{alg:res_pseudospec2}. These results align in areas where the variance is minimal (large $|\lambda|$). However, in regions where $|\lambda|$ is small, the variance
component in \eqref{eq:res_decomp} becomes significant. This observation leads us to infer that only about seven eigenpairs are of meaningful significance in a statistically coherent framework.

\begin{figure}
\centering
\includegraphics[width=0.4\textwidth]{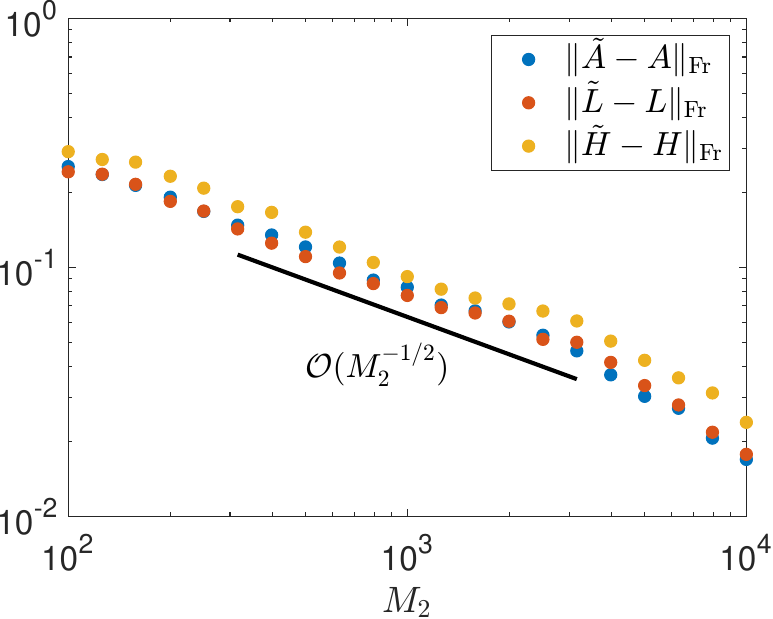}
\caption{Estimation error for the matrices $\tilde{A},\tilde{L}$ and $\tilde{H}$ for the circle map. The solid line shows the expected Monte--Carlo convergence rate.}
\label{fig:circle_con}
\end{figure}

\begin{figure}
\centering
\includegraphics[width=0.4\textwidth]{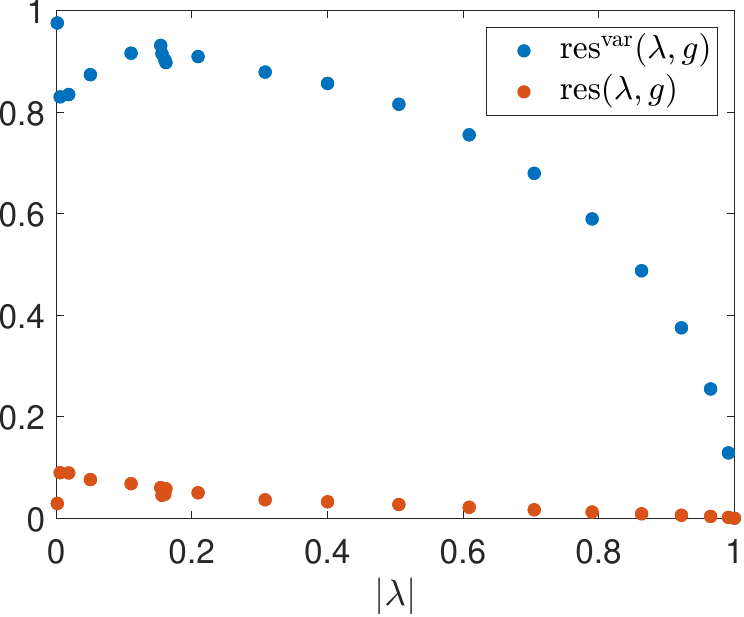}
\caption{Residuals for the circle map computed using Algorithm \ref{alg:stoch_resDMD2}.}
\label{fig:circle_res}
\end{figure}

\begin{figure}
\centering
\includegraphics[width=0.4\textwidth]{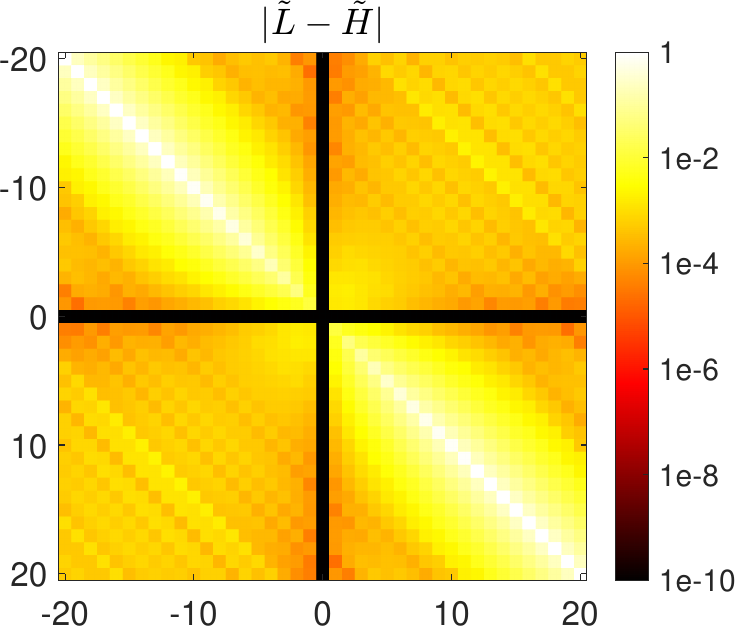}
\caption{Absolute values of the matrix $\tilde{L}-\tilde{H}$ for the circle map. This difference corresponds to the covariance matrix in \eqref{cov_mat_def}.}
\label{fig:circle_cov}
\end{figure}

\begin{figure*}
\centering
\includegraphics[width=0.48\textwidth]{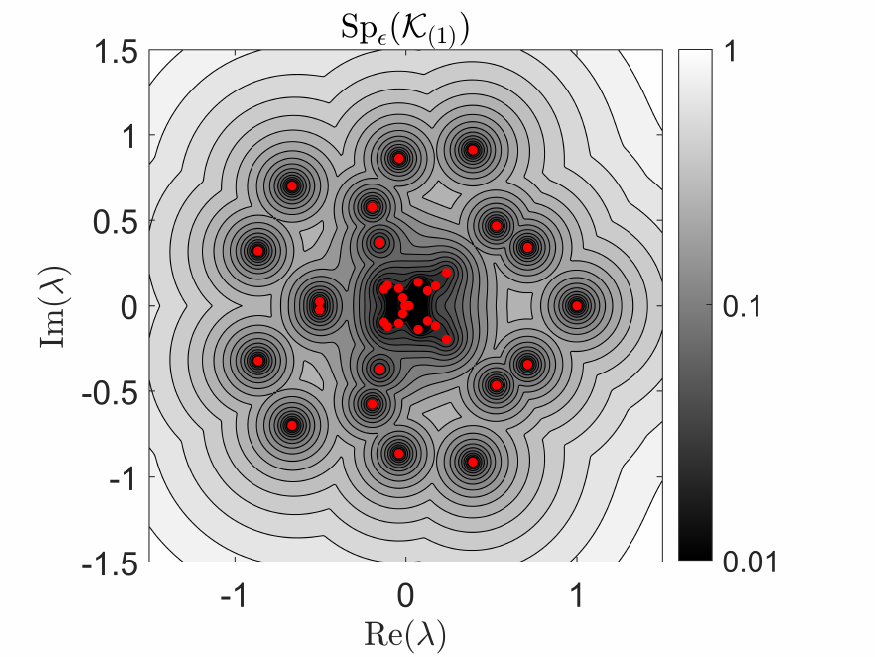}
\hfill
\includegraphics[width=0.48\textwidth]{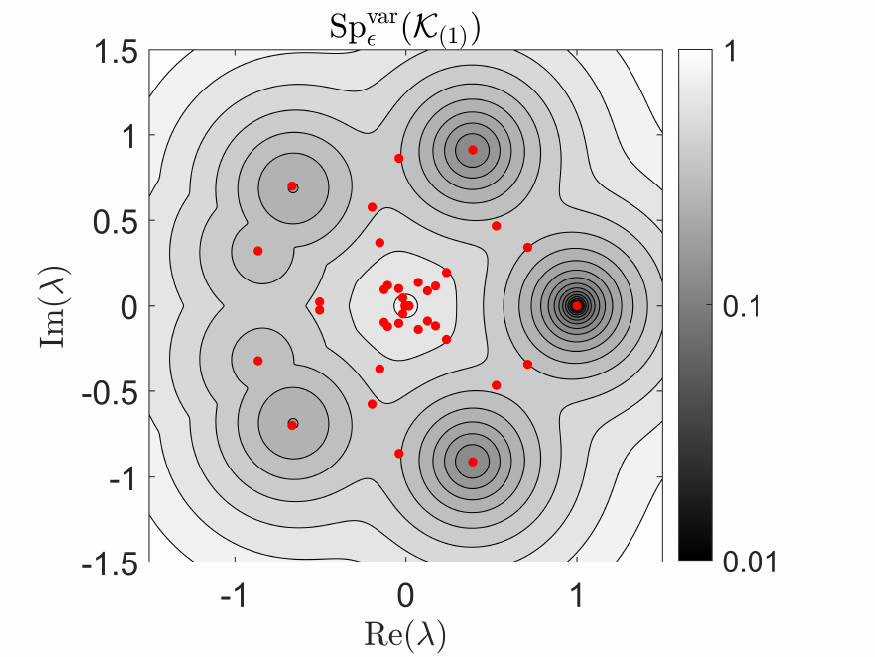}
\caption{Pseudospectra vs. variance pseudospectra. Left: Output of Algorithm \ref{alg:res_pseudospec1} for the circle map. Right: Output of Algorithm \ref{alg:res_pseudospec2} for the circle map. We have shown the minimized residuals over a contour plot of $\epsilon$ in both cases. The red dots correspond to the EDMD eigenvalues.}
\label{fig:circle_pseudospec}
\end{figure*}

\subsection{Stochastic Van der Pol oscillator}
\label{sec:stoch_vdp}

\begin{figure*}
\centering
\includegraphics[width=0.48\textwidth]{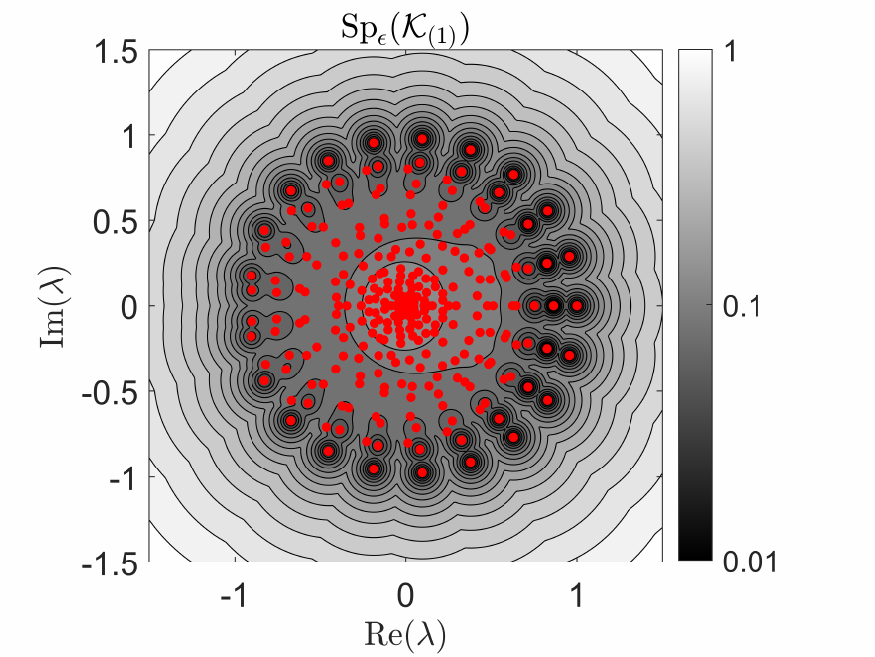}
\hfill
\includegraphics[width=0.48\textwidth]{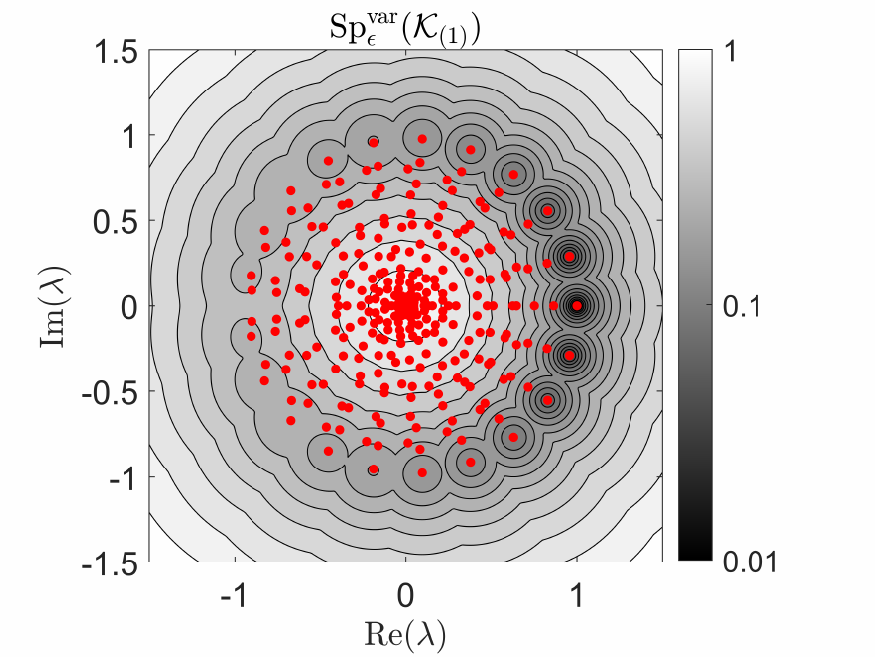}
\caption{Pseudospectra vs. variance pseudospectra. Left: Output of Algorithm \ref{alg:res_pseudospec1} for the stochastic Van der Pol oscillator. Right: Output of Algorithm \ref{alg:res_pseudospec2} for the stochastic Van der Pol oscillator. We have shown the minimized residuals over a contour plot of $\epsilon$ in both cases. The red dots correspond to the EDMD eigenvalues.}
\label{fig:SDE_pseudospec}
\end{figure*}

We now consider the stochastic differential equation
\begin{align*}
\mathrm{d} X_1 &= X_2 \mathrm{d}t\\
\mathrm{d}X_2 &= \left[\mu(1-X_1^2)X_2-X_1\right] \mathrm{d}t +\sqrt{2\delta}\mathrm{d} B_t,
\end{align*}
where $B_t$ denotes standard one-dimensional Brownian motion, $\delta>0$, and $\mu>0$.\footnote{The inclusion of Brownian motion only in the $\mathrm{d}X_2$ term is motivated by the physical interpretation of the random driving force. However, adding a similar term to the $\mathrm{d}X_1$ equation would only affect the Kolmogorov operator by altering the parameter $\delta$.} This equation represents a noisy version of the Van der Pol oscillator. In the absence of noise, the Van der Pol oscillator exhibits a limit cycle to which all initial conditions converge, except for the unstable fixed point at the origin. The introduction of noise transforms the system, resulting in a global attractor that forms a band around the deterministic system's limit cycle.

\begin{figure*}
\centering
\includegraphics[width=0.32\textwidth]{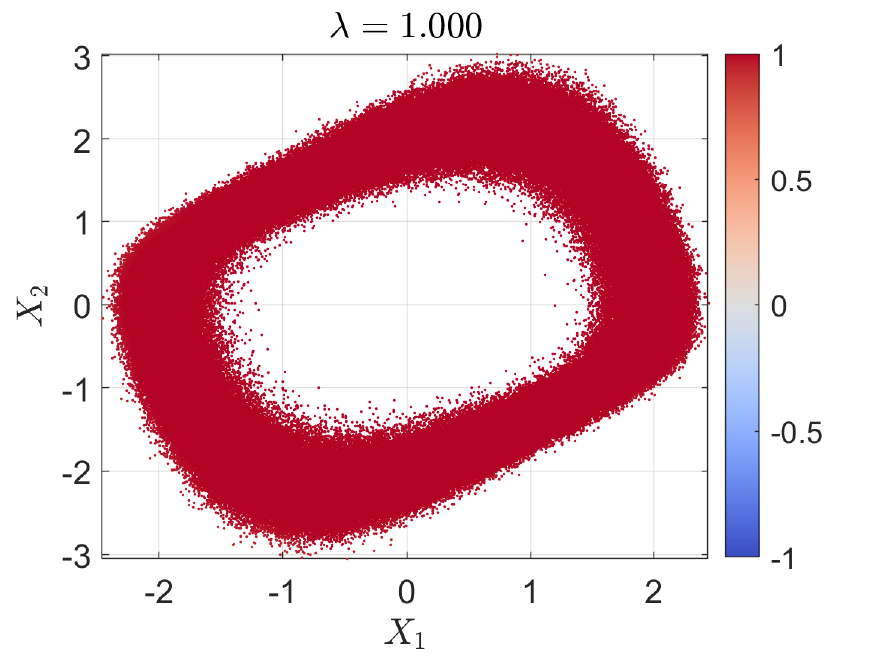}
\includegraphics[width=0.32\textwidth]{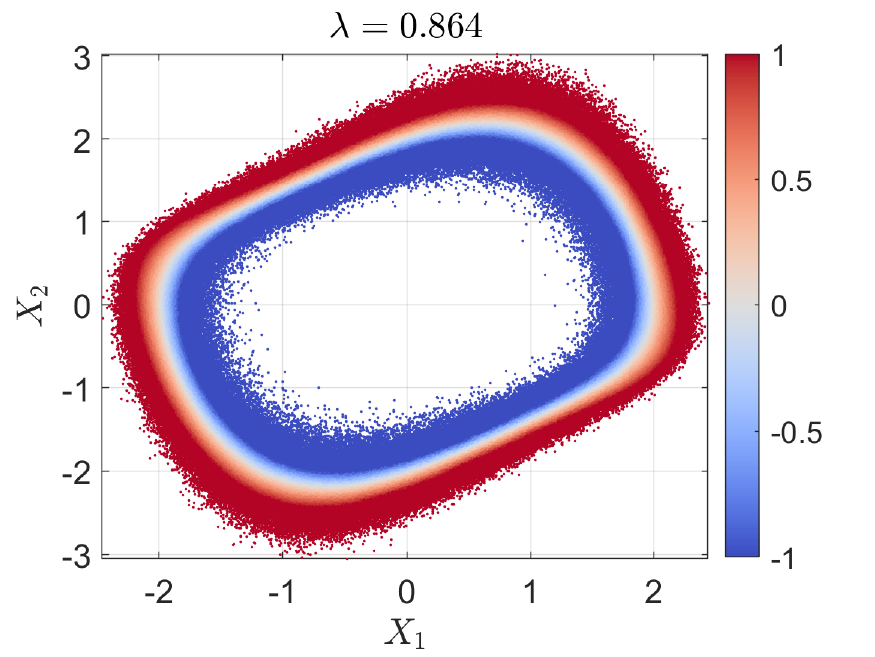}
\includegraphics[width=0.32\textwidth]{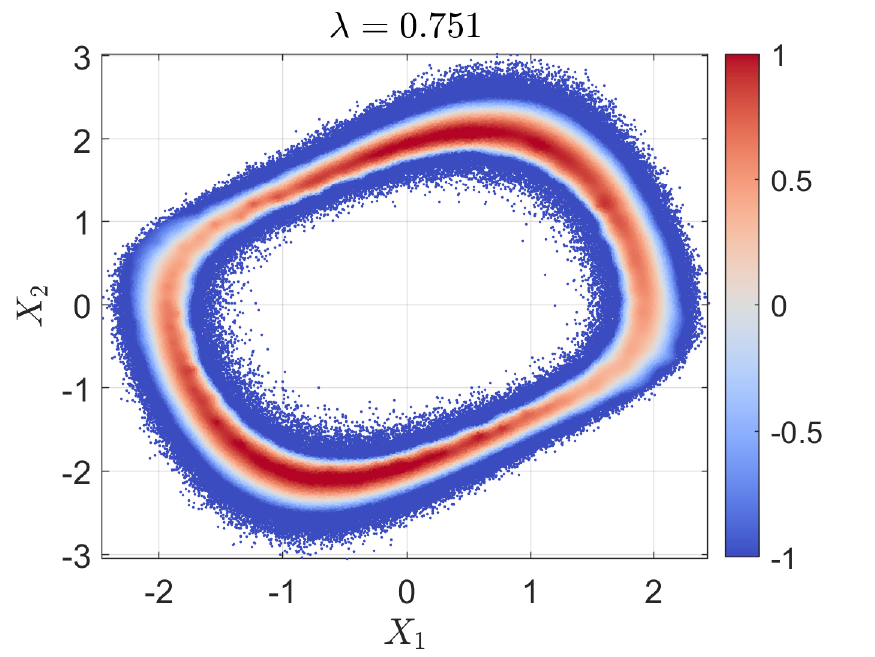}\\
\includegraphics[width=0.32\textwidth]{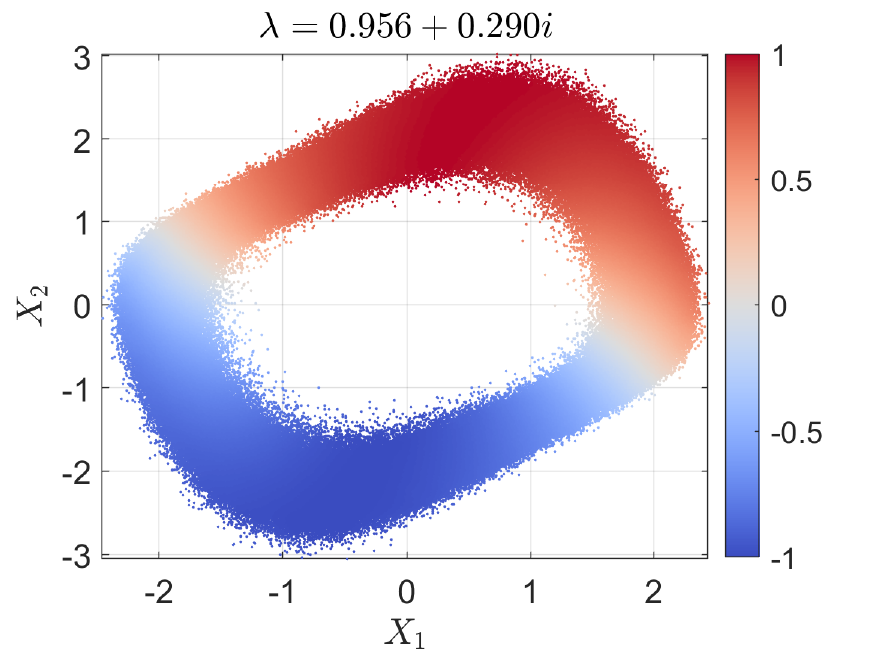}
\includegraphics[width=0.32\textwidth]{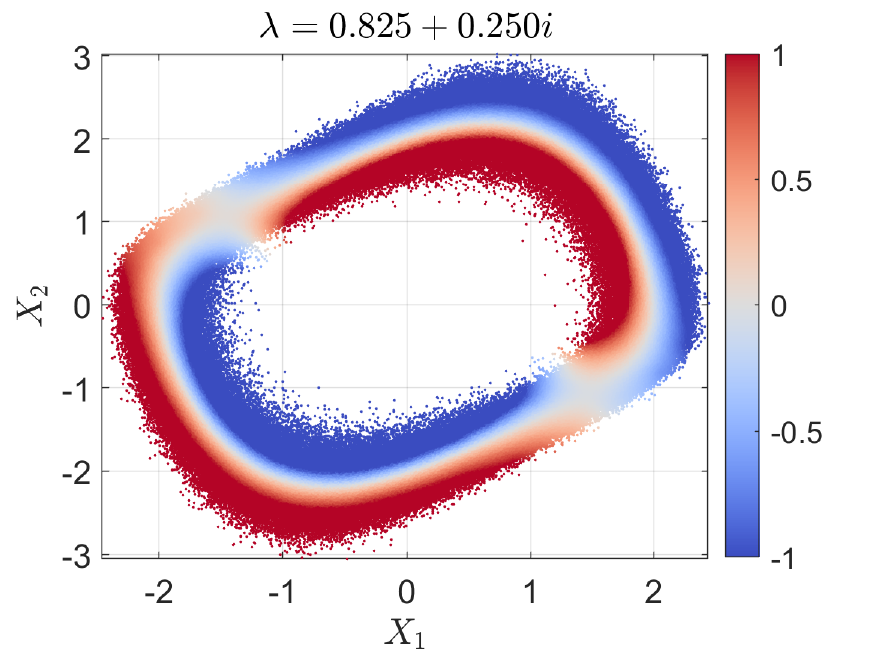}
\includegraphics[width=0.32\textwidth]{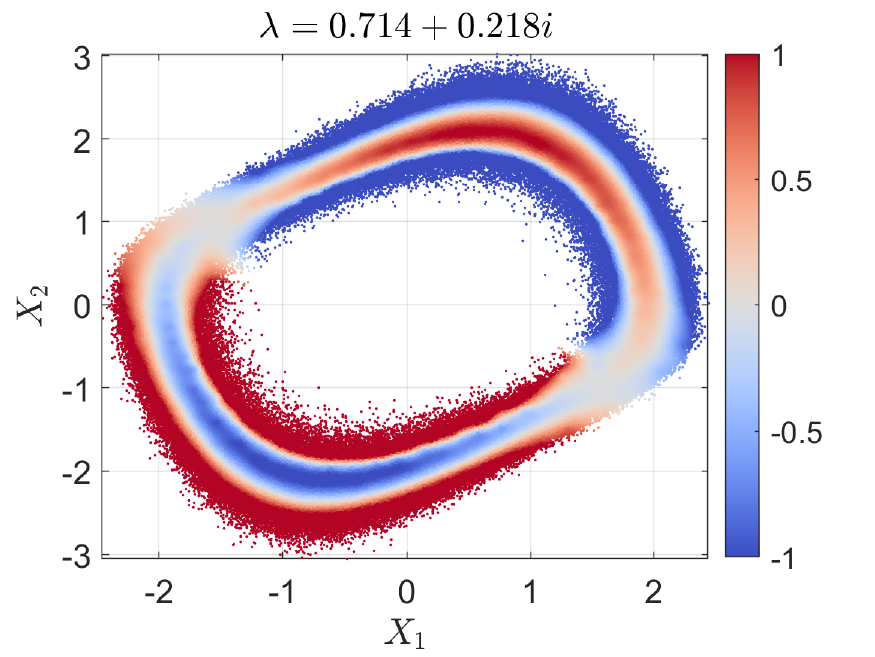}\\
\includegraphics[width=0.32\textwidth]{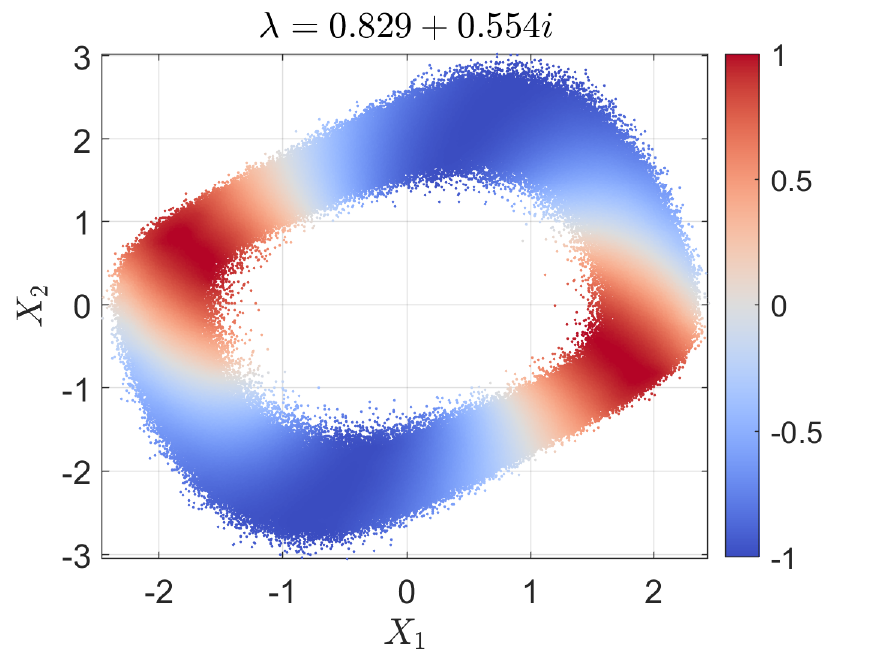}
\includegraphics[width=0.32\textwidth]{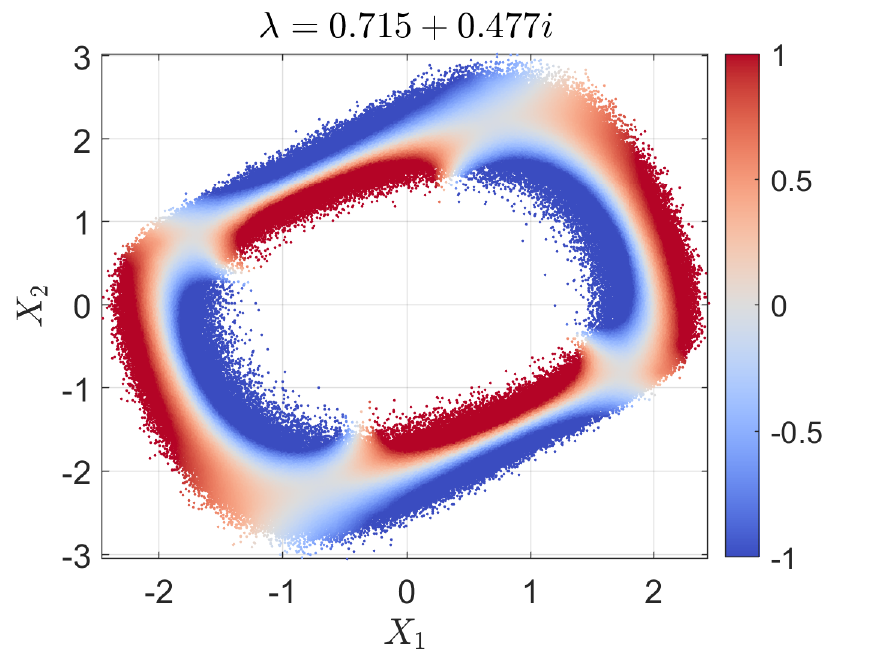}
\includegraphics[width=0.32\textwidth]{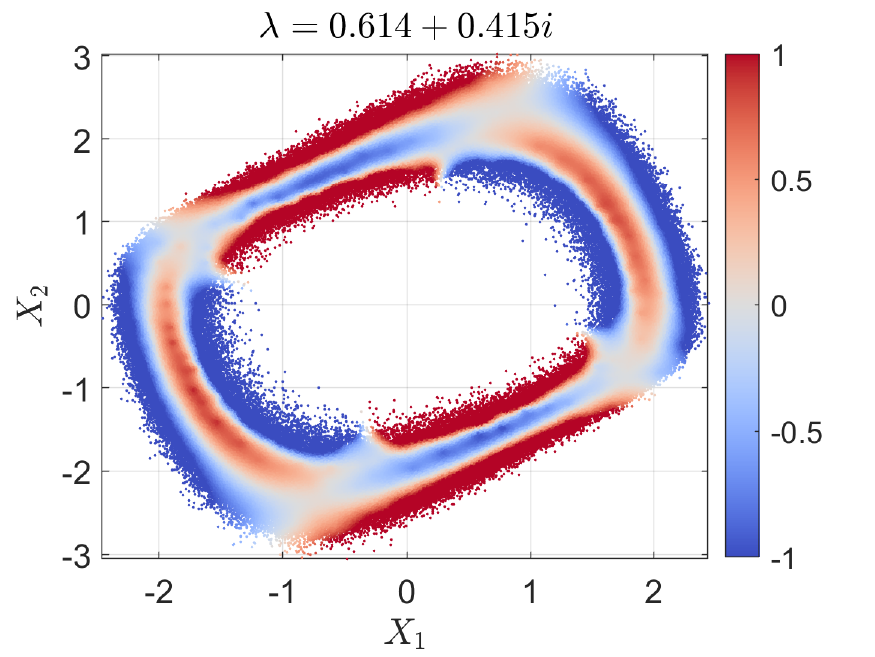}\\
\includegraphics[width=0.32\textwidth]{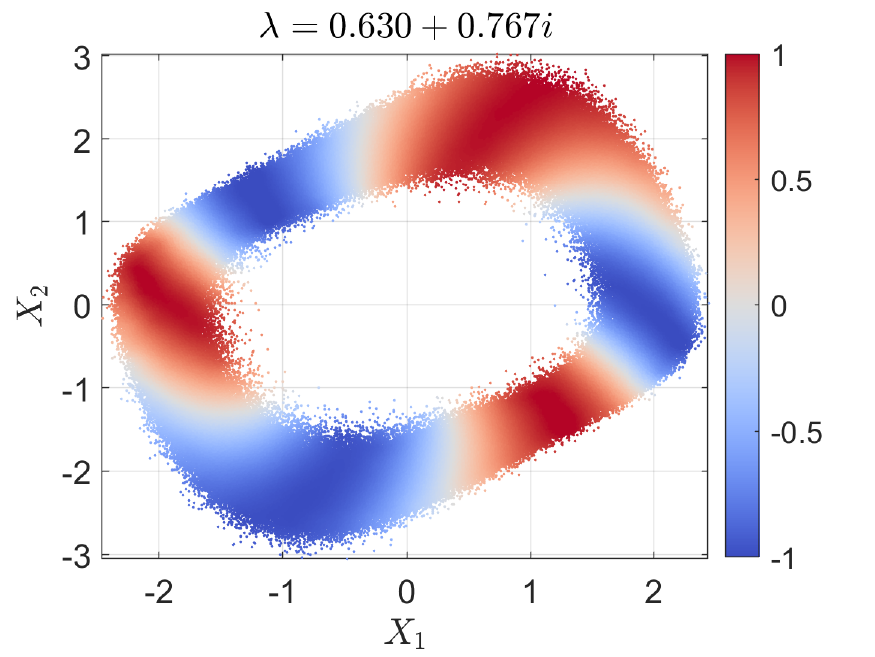}
\includegraphics[width=0.32\textwidth]{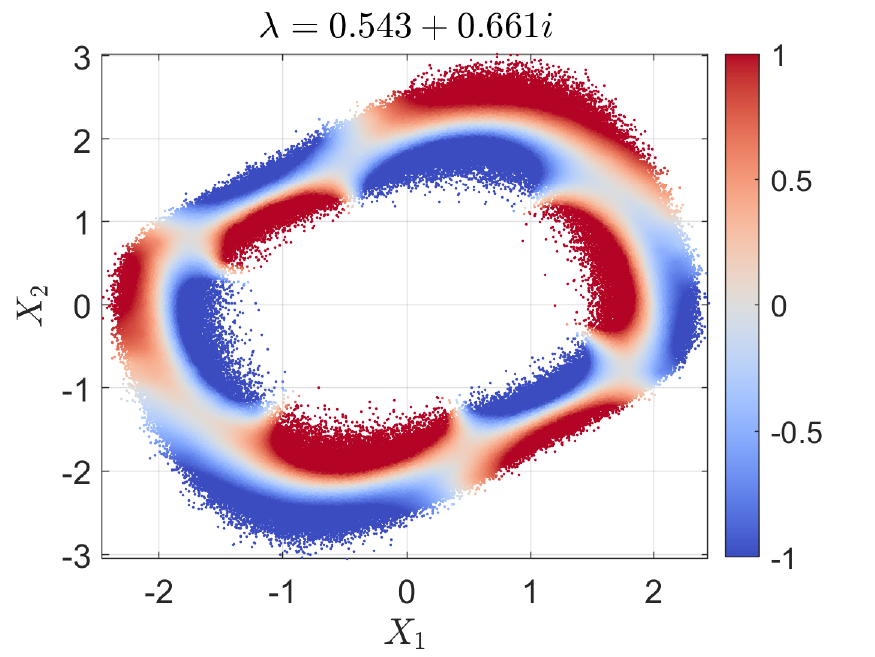}
\includegraphics[width=0.32\textwidth]{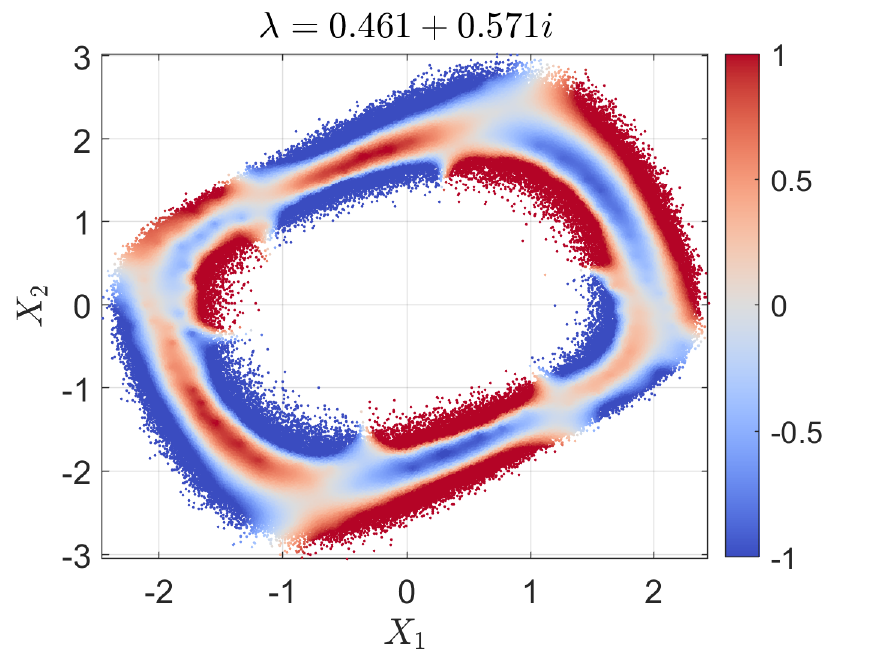}\\
\includegraphics[width=0.32\textwidth]{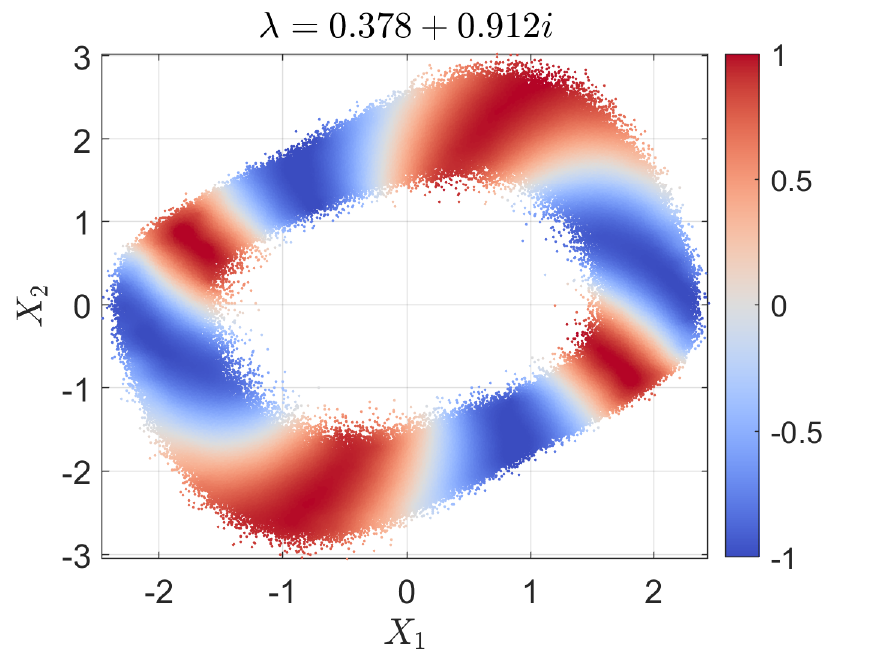}
\includegraphics[width=0.32\textwidth]{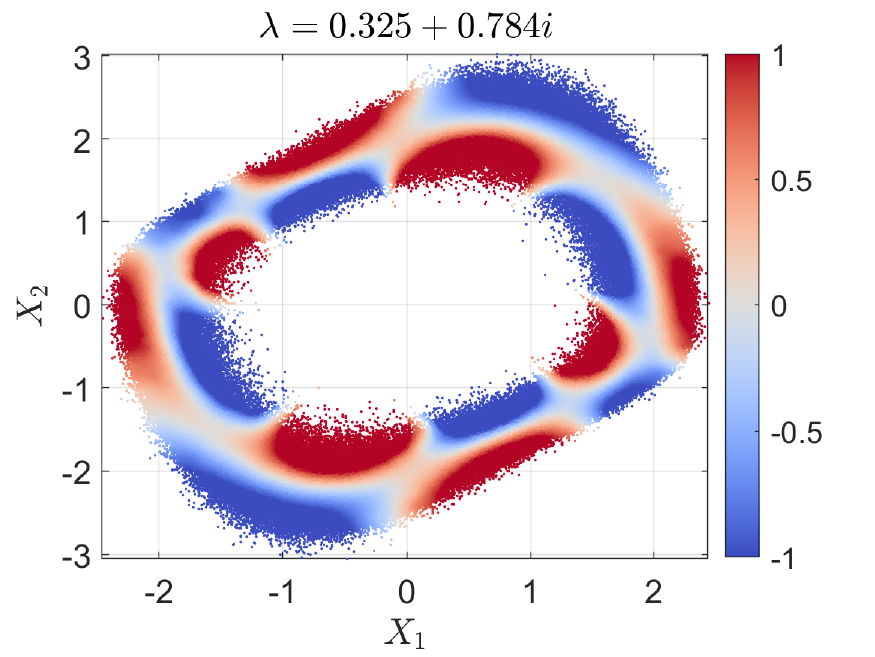}
\includegraphics[width=0.32\textwidth]{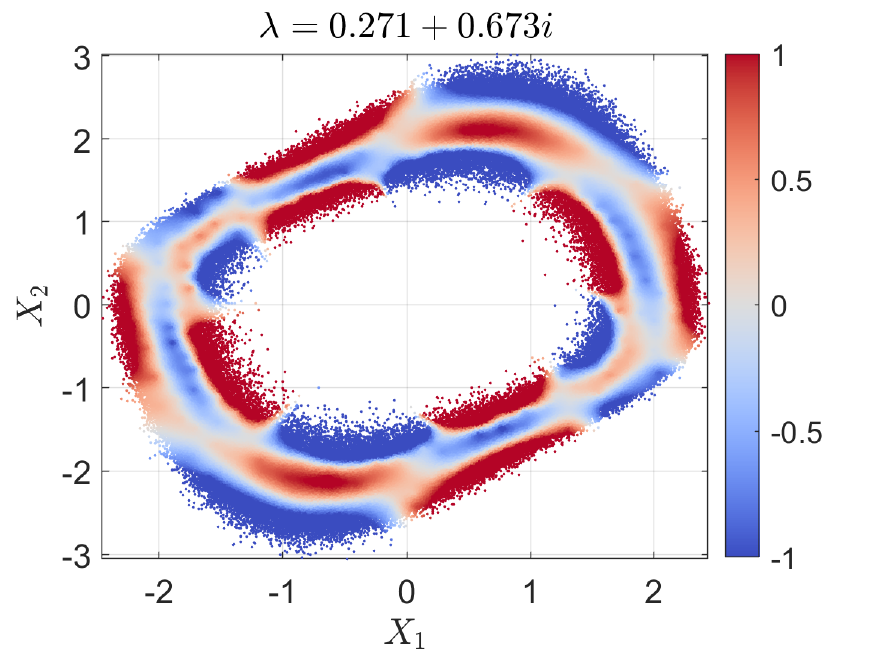}
\caption{Computed eigenfunctions (real part shown) of the stochastic Van der Pol oscillator. Due to conjugate symmetry, we have only shown eigenfunctions corresponding to eigenvalues with non-negative imaginary parts.}
\label{fig:SDE_modes}
\end{figure*}

The generator of the stochastic solutions, known as the backward Kolmogorov operator, is described in \cite[Section 9.3]{da2014stochastic}. It is a second-order elliptic type differential operator $\mathcal{L}$, defined by
\begin{align*}
[\mathcal{L}g](X_1,X_2) &=  \begin{pmatrix}
\!\pmb{x}_2\\ \mu(1-X_1^2)X_2-X_1
\end{pmatrix} \cdot\nabla g(X_1,X_2)\\
&\quad\quad\quad\quad\quad\quad +\delta \nabla^2g(X_1,X_2).
\end{align*}
For a discrete times step $\Delta_t$, the Koopman operator is given by $\exp(\Delta_t \mathcal{L})$. In the absence of noise ($\delta=0$), the Koopman operator has eigenvalues forming a lattice \cite[Theorem 13]{mezic2017koopman}:
$$
\left\{\hat{\lambda}_{m,k}=\exp([-m\mu + ik\omega_0]\Delta_t):k\in\mathbb{Z},m\in\mathbb{N}\cup\{0\}\right\},
$$
where $\omega_0\approx1-\mu^2/16$ is the base frequency of the limit cycle \cite{strogatz2018nonlinear}. When 
$\delta$ is moderate, the base frequency of the averaged limit cycle remains similar to that in the deterministic case \cite{leung1995stochastic}.

We simulate the dynamics using the Euler--Maruyama method \cite{rossler2010runge} with a time step of $3\times 10^{-3}$. Data are collected along a single trajectory of length $M_1=10^6$ with $M_2=2$, starting the sampling after the trajectory reaches the global attractor. We employ 318 Laplacian radial basis functions with centers on the attractor as our dictionary. The parameters are set to $\mu=0.5$, $\delta=0.02$, and $\Delta_t=0.3$.

\begin{figure*}
\centering
\includegraphics[width=0.4\textwidth]{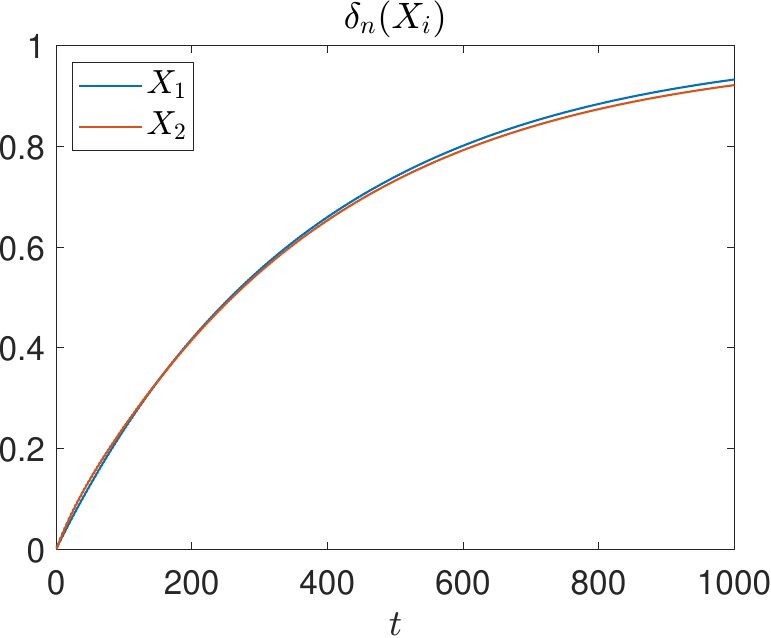}
\includegraphics[width=0.4\textwidth]{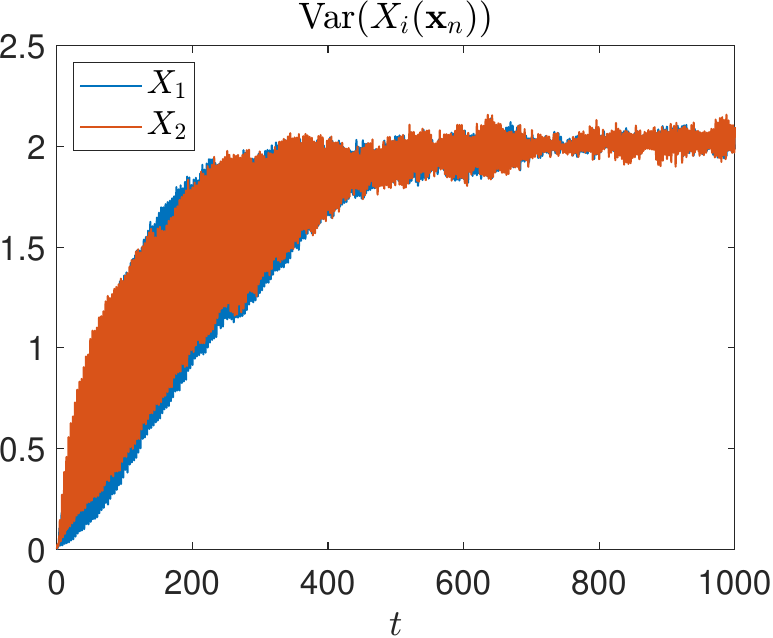}
\caption{Left: Subspace errors $\delta_n(X_1)$ and $\delta_n(X_2)$ for the stochastic Van der Pol oscillator, computed using \eqref{delta_comp1} and \eqref{delta_comp2}. Right: Variance of trajectory. We have rescaled the horizontal axis in both plots to correspond to time.}
\label{fig:SDE_var}
\end{figure*}

Figure \ref{fig:SDE_pseudospec} displays the results obtained using Algorithms \ref{alg:res_pseudospec1} and \ref{alg:res_pseudospec2}. Similar to observations from the circle map example, 
$\mathrm{Sp}_\epsilon(\mathcal{K}_{(1)})$ and $\mathrm{Sp}_\epsilon^\mathrm{var}(\mathcal{K}_{(1)})$ exhibit greater similarity near the unit circle. The lattice-like structure in the eigenvalues is also evident, with the EDMD-computed eigenvalues appearing as perturbations of the set $\{\hat{\lambda}_{m,k}\}$. Table \ref{tab:SDEtab} lists some of these eigenvalues alongside the residuals calculated using Algorithm \ref{alg:stoch_resDMD2}. We observe that as $|k|$ increases, $\mathrm{res}(\lambda,g)$ also increases, and similarly, $\mathrm{res}^{\mathrm{var}}(\lambda,g)$ increases with $m$. For any given eigenvalue, $\mathrm{res}(\lambda,g)$ decreases to zero with larger dictionaries. In contrast, $\mathrm{res}^{\mathrm{var}}(\lambda,g)$ approaches a finite non-zero value, except for the trivial eigenvalue, which has a constant eigenfunction exhibiting zero variance. Figure \ref{fig:SDE_modes} illustrates the corresponding eigenfunctions on the attractor, showcasing their beautiful modal structure.

\begin{table}
\caption{Computed eigenvalues of the stochastic Van der Pol oscillator, and the residuals computed using Algorithm \ref{alg:stoch_resDMD2}. We have ordered them according to perturbations of $\hat{\lambda}_{m,k}$. Due to conjugate symmetry, we have only shown eigenvalues with non-negative imaginary parts.}
\label{tab:SDEtab}       
\centering
\begin{tabular}{rcrcc}
\hline\noalign{\smallskip}
$\lambda\approx \hat{\lambda}_{m,k}$ & $m$ & $k$ &$\mathrm{res}^{\mathrm{var}}$ & $\mathrm{res}$ \\
\noalign{\smallskip}\hline\noalign{\smallskip}
$1.000 + 0.000i$  & $0$ & $0$  & $0.001$ & $0.001$  \\ 
$0.956 + 0.290i$  & $0$ & $1$  & $0.040$ & $0.001$  \\ 
$0.829 + 0.554i$  & $0$ & $2$  & $0.080$ & $0.002$ \\
$0.630 + 0.767i$  & $0$ & $3$  & $0.120$ & $0.005$ \\
$0.378 + 0.912i$  & $0$ & $4$  & $0.159$ & $0.008$  \\
$0.096 + 0.975i$  & $0$ & $5$  & $0.198$ & $0.012$ \\
$-0.190 + 0.953i$ & $0$ & $6$  & $0.237$ & $0.016$ \\ 
$-0.454 + 0.848i$ & $0$ & $7$  & $0.275$ & $0.022$   \\ 
$-0.672 + 0.671i$ & $0$ & $8$  & $0.313$ & $0.029$  \\ 
$0.864 + 0.000i$  & $1$ & $0$  & $0.504$ & $0.017$   \\ 
$0.825 + 0.250i$  & $1$ & $1$  & $0.506$ & $0.009$  \\
$0.715 + 0.477i$  & $1$ & $2$  & $0.511$ & $0.013$   \\
$0.543 + 0.661i$  & $1$ & $3$  & $0.518$ & $0.024$   \\ 
$0.325 + 0.784i$  & $1$ & $4$  & $0.528$ & $0.033$   \\ 
$0.083 + 0.838i$  & $1$ & $5$  & $0.541$ & $0.041$   \\ 
$-0.163 + 0.816i$ & $1$ & $6$  & $0.555$ & $0.051$  \\ 
$-0.388 + 0.724i$ & $1$ & $7$  & $0.571$ & $0.062$   \\ 
$-0.572 + 0.571i$ & $1$ & $8$  & $0.589$ & $0.074$   \\ 
$0.751 + 0.000i$  & $2$ & $0$  & $0.661$ & $0.057$   \\ 
$0.714 + 0.218i$  & $2$ & $1$  & $0.665$ & $0.066$   \\ 
$0.614 + 0.415i$  & $2$ & $2$  & $0.671$ & $0.075$   \\ 
$0.461 + 0.571i$  & $2$ & $3$  & $0.679$ & $0.084$  \\ 
$0.271 + 0.673i$  & $2$ & $4$  & $0.689$ & $0.094$   \\ 
$0.061 + 0.712i$  & $2$ & $5$  & $0.700$ & $0.104$   \\ 
$-0.149 + 0.685i$ & $2$ & $6$  & $0.713$ & $0.117$  \\ 
$-0.336 + 0.597i$ & $2$ & $7$  & $0.729$ & $0.131$  \\
$-0.550 + 0.463i$ & $2$ & $8$  & $0.696$ & $0.144$ \\ 
\noalign{\smallskip}\hline
\end{tabular}
\end{table}

\begin{figure*}
\centering
\includegraphics[width=0.49\textwidth]{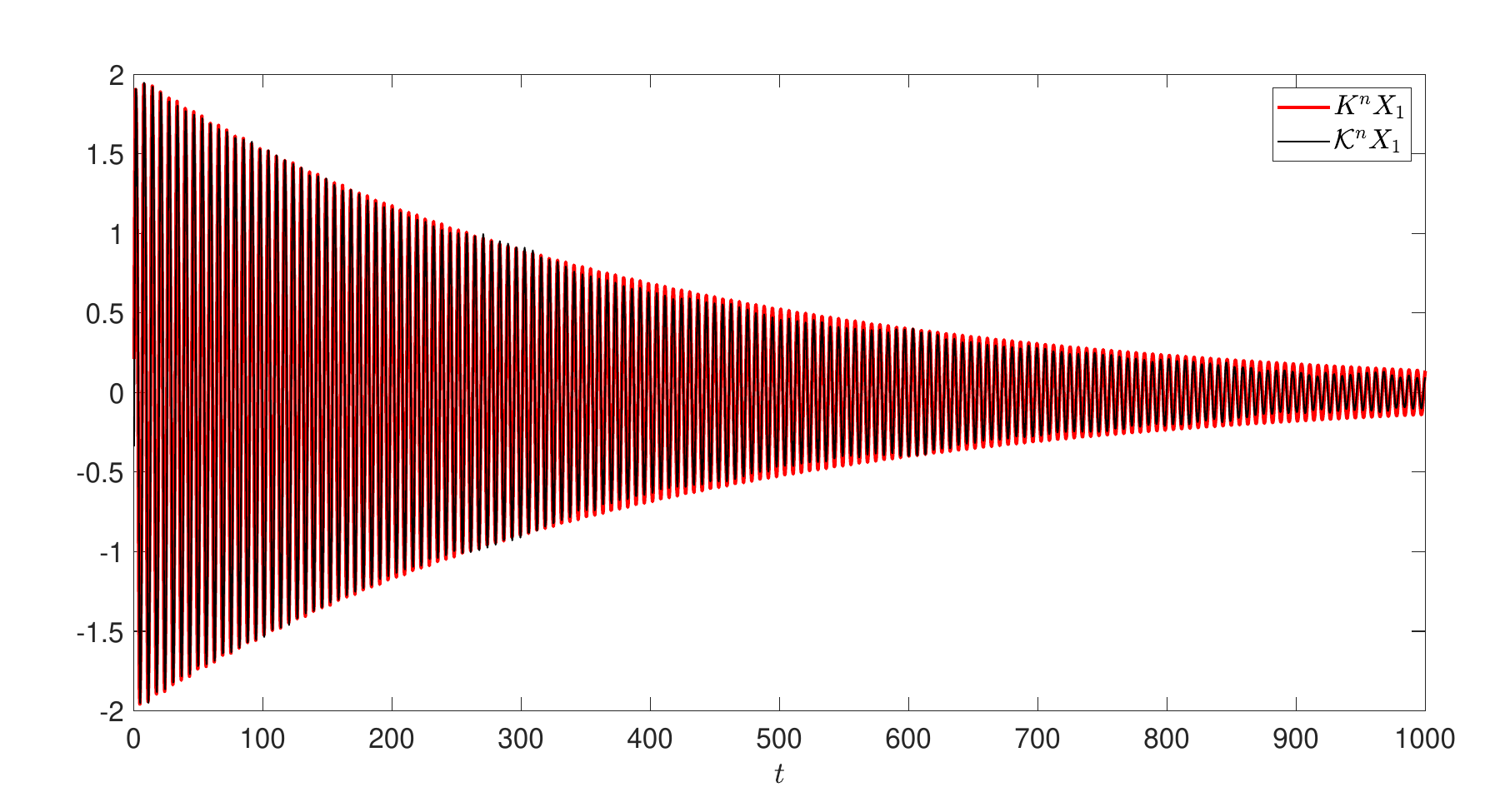}
\includegraphics[width=0.49\textwidth]{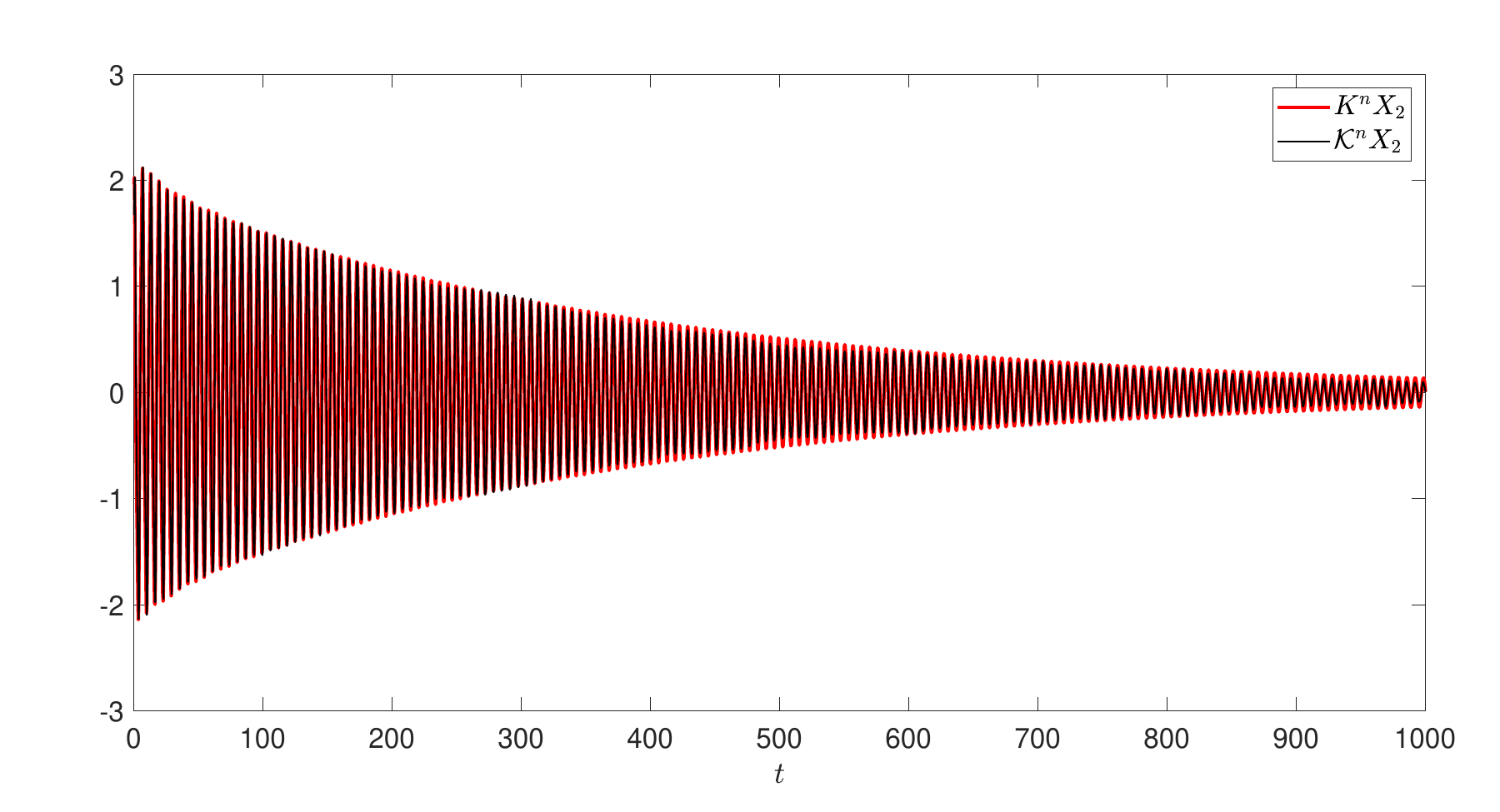}
\caption{Comparison of computed $K^nX_i$, where $K\in\mathbb{C}^{N\times N}$ is the EDMD matrix, and the true values of $\mathcal{K}^nX_i$.}
\label{fig:SDE_pred}
\end{figure*}

In this example, the norm of the Koopman operator $\|\mathcal{K}\|$ is approximately 1, and the subspace error $\delta_n(g)$ predominantly contributes to the bound established in Theorem \ref{thm:forecast_bound}. We analyze the two observables $X_1$ and $X_2$, each starting from a point randomly selected on the attractor. Figure \ref{fig:SDE_var} presents the calculated values of $\delta_n(X_1)$ and $\delta_n(X_2)$ as per \eqref{delta_comp1} and \eqref{delta_comp2}, along with the variance of the trajectory. Additionally, Figure \ref{fig:SDE_pred} compares the values computed using $K^nX_i$ with the actual values of $\mathcal{K}^nX_i$, obtained by integrating the generator $\mathcal{L}$. Together, these figures demonstrate the convergence of the mean trajectories towards the dominant subspace of $\mathcal{K}$.

\subsection{Neuronal population dynamics}

As a final example, we apply our approach to experimental neuroscience data. Recent technological advancements in this field now allow for the simultaneous monitoring of large neuronal populations in the brains of awake, behaving animals. This development has spurred significant interest in employing data-driven methods to derive physically meaningful insights from high-dimensional neural measurements \cite{paninski2018neural}.

To analyze complex neural data, researchers have employed a variety of analytical tools to uncover features like low-dimensional manifolds, latent population dynamics, within-trial variance, and trial-to-trial variability. However, existing methods often examine these features in isolation \cite{gao2016linear,churchland2010stimulus,pandarinath2018inferring,stringer2019spontaneous}. From a dynamical systems perspective, a unified model that captures these distinct aspects of neural data would be highly advantageous. In this context, the Koopman operator framework offers a compelling approach to analyzing high-dimensional neural observables \cite{marrouch2020data}. DMD has emerged as a prominent method for the spatiotemporal decomposition of diverse datasets \cite{brunton2016extracting,casorso2019dynamic}. Nevertheless, a limitation of DMD is its lack of explicit uncertainty quantification regarding the modes and forecasts it uncovers. This aspect is particularly vital in neural time series analysis, where it is challenging to identify physically meaningful spectral components \cite{donoghue2020parameterizing}.

Our framework offers a unified, data-driven solution to uncover validated latent dynamical modes and their associated variance in neural data. To demonstrate its efficacy, we applied it to high-dimensional neuronal recordings from the visual cortex of awake mice, as publicly shared by the Allen Brain Observatory \cite{siegle2021survey}, involving 400–800 neurons per mouse. Our focus was on the ``Drifting Gratings" task epoch, wherein mice were presented with gratings drifting in one of eight directions (0$^{\circ}$, 45$^{\circ}$, etc.), modulated sinusoidally at one of five temporal frequencies. We specifically analyzed responses to gratings modulated at 15 Hz across all eight directions, as these stimuli consistently elicited an identifiable eigenvalue in the neural data corresponding to the expected frequency. This analysis encompassed 120 trials per mouse (stimulus duration of 2s) for a total of 20 mice, as detailed in \cite{siegle2021survey}. We computed distinct stochastic Koopman operators for 15 different arousal levels, categorized by the average pupil diameter measured during the 500ms before each stimulus \cite{mcginley2015cortical}. For this analysis, DMD was employed to identify 100 dictionary functions.

Our data-driven approach was effective in identifying an isolated, population-level coherent mode at the stimulus frequency. As illustrated in Figure \ref{fig:mouse_pseudospec}, this is evidenced by a distinct eigenvalue, highlighted in green, which consistently appears as a clear local minimum in the variance pseudospectra contour plots across various arousal states. Without the variance pseudospectra, discerning which DMD eigenvalues are reliable and indicative of coherence can be challenging. We observed that individual neurons displayed a variety of waveforms, all linked to this single linear dynamic mode. Demonstrating the diversity of these responses, Figure \ref{fig:mouse_traj} showcases five randomly chosen sample trajectories from the KMD. These trajectories highlight the distinct spike counts and/or timings of different neurons, all parsimoniously represented by a single latent mode.

\begin{figure*}
\centering
\includegraphics[width=0.33\textwidth]{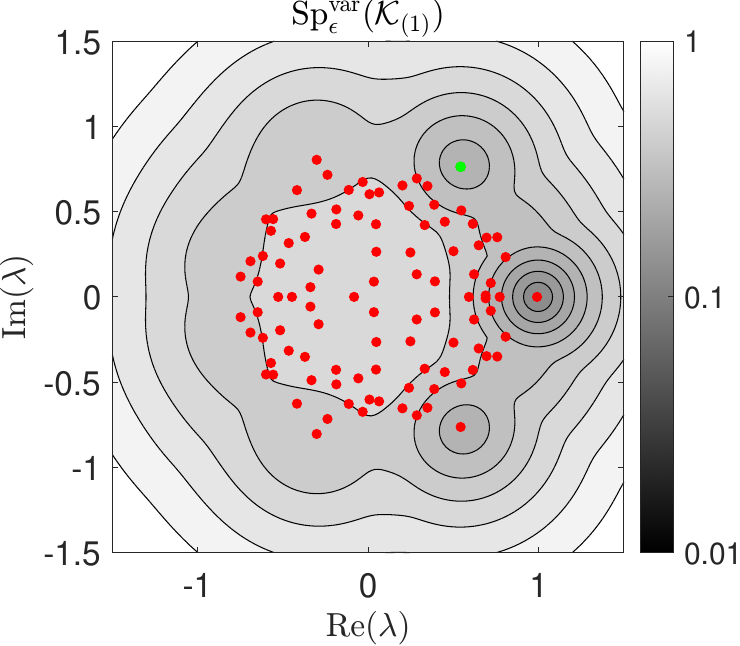}
\hfill
\includegraphics[width=0.33\textwidth]{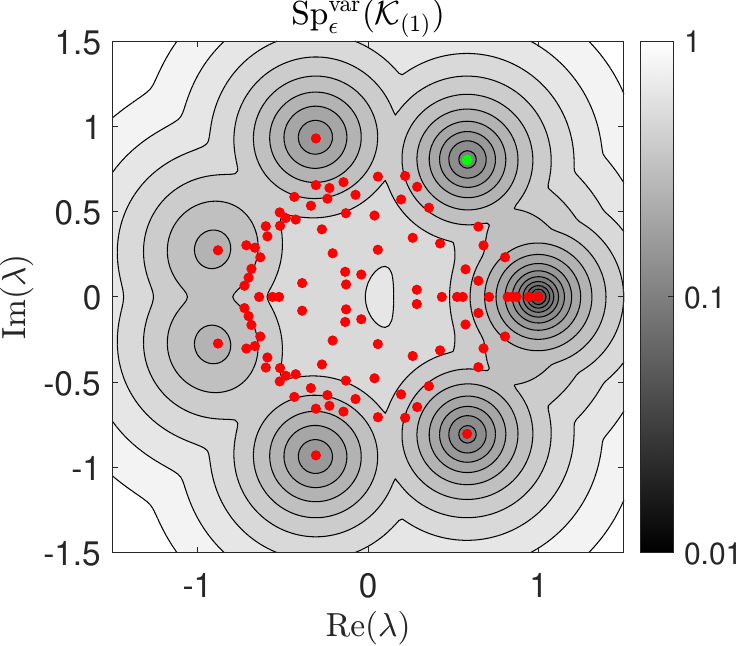}
\hfill
\includegraphics[width=0.33\textwidth]{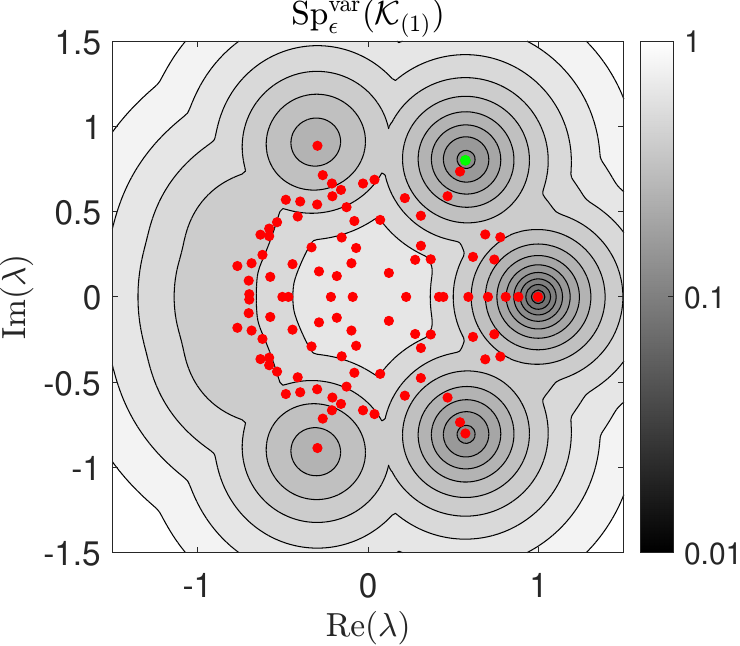}
\caption{Variance pseudospectra for a single mouse in the neuronal population dynamics example. Each case corresponds to a pupil diameter of $8\%$ (left), $28\%$ (middle), and $43\%$ (right). The identified mode is shown in green, and the red dots show the other DMD eigenvalues. The variance pseudospectra changes considerably as the arousal state changes, but the green eigenvalue shows little variability.}
\label{fig:mouse_pseudospec}
\end{figure*}

\begin{figure}
\centering
\includegraphics[width=0.4\textwidth]{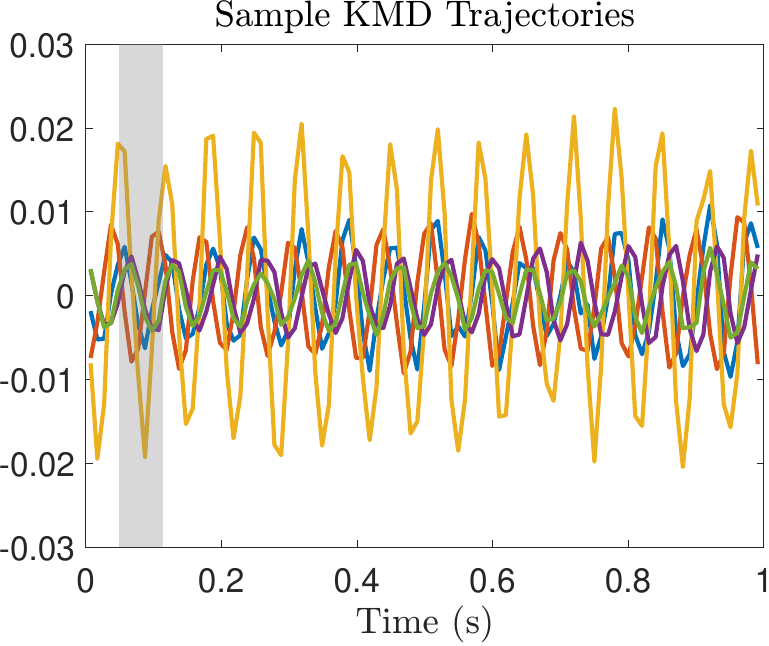}
\caption{Randomly selected sample trajectories from the Koopman mode corresponding to the eigenvalue shown in green in Figure \ref{fig:mouse_pseudospec}. The reference gray region in the left region shows the wavelength predicted by the eigenvalue.}
\label{fig:mouse_traj}
\end{figure}

Importantly, neuronal responses demonstrate significant trial-to-trial variability, a phenomenon of considerable physiological interest due to its close relationship with ongoing fluctuations in an animal's internal state. Dynamical systems approaches are adept at modeling this type of variability, which often stems from changes in the neural population's pre-stimulus state \cite{pandarinath2018inferring}. Furthermore, the extent of this variability is heavily influenced by internal states like arousal and attention, as detailed in \cite{mcginley2015waking}. Our stochastic modeling approach enables us to additionally estimate this second source of trial-to-trial variability in neuronal responses.

To validate the physiological significance of our variance estimates, we analyzed the variance linked to the Koopman operators computed across each of $15$ levels of pupil diameter, effectively using pupil diameter as a parameter for the Koopman operator in relation to arousal. Our hypothesis was that this analysis would reflect the well-known ``U-shape" pattern described by the Yerkes--Dodson law \cite{yerkes1908relation}, with variance minimized at intermediate arousal levels \cite{mcginley2015cortical}. Figure \ref{fig:mouse_pseudospec} indicates that the eigenvalue or expectation derived from \ref{fig:mouse_pseudospec} remains consistent across various arousal states. However, from Figure \ref{fig:mouse_var}, a notable modulation in variance residuals is observed in accordance with arousal levels, aligning with our predictions: the variance associated with the leading mode is specifically reduced at intermediate arousal levels. This pattern underscores the physiological relevance of the variance estimates yielded by our modeling approach. Consequently, our findings suggest that arousal systematically influences dynamical variance, providing both practical and physiological rationales for employing dynamical models that explicitly estimate variance. Overall, our data-driven framework offers a unified and formal representation of neural dynamics, parsimoniously capturing multiple physiologically significant features in the data.

\begin{figure}
\centering
\includegraphics[width=0.4\textwidth]{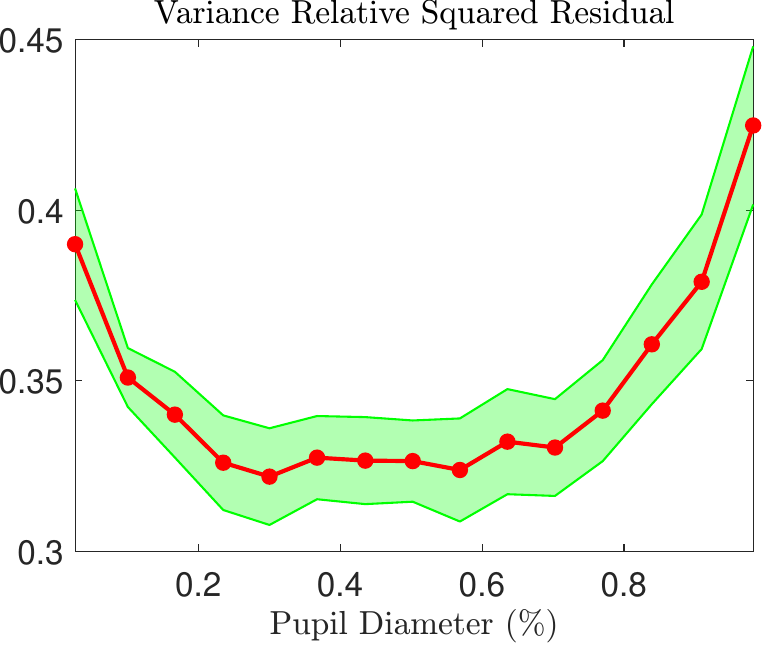}
\caption{The variance relative squared residual as a function of the arousal state. The red lines show the average across the mice, and the green error bounds correspond to the standard error of the mean. The ``U-shape" is characteristic of the so-called Yerkes--Dodson law, which we produce in a data-driven fashion from the dynamics.}
\label{fig:mouse_var}
\end{figure}

\section{Conclusion}

We have demonstrated the role of variance in the Koopman analysis of stochastic dynamical systems. To effectively study projection errors in data-driven approaches for these systems, it is crucial to move beyond expectations and study more than just the stochastic Koopman operator. Incorporating variance into the Koopman framework enhances our understanding of spectral properties and the related projection errors. By analyzing various types of residuals, we have developed data-driven algorithms capable of computing the spectral properties of infinite-dimensional stochastic Koopman operators. Furthermore, we introduced the concept of variance pseudospectra, a tool designed to assess statistical coherency. From a computational perspective, our work includes several convergence theorems pertinent to the spectral properties of these operators. In the realm of experimental neural recordings, our framework has proven effective in extracting and compactly representing multiple data features with known physiological significance.

There are several avenues of future work related to this paper. One such direction involves an analysis of the algorithms and theorems presented in Section \ref{sec:theory} in scenarios involving noisy snapshot data. Another avenue explores the trade-offs between computing the squared residual and variance terms, as outlined in \eqref{IMSE_decomp}, potentially reflecting variance-bias trade-offs in statistical analysis. Additionally, we aim to assess the robustness and generalizability of the proposed framework across further stochastic dynamical systems.

\begin{acknowledgements}
We thank the Allen Institute for the publicly available data and the referees for valuable comments that helped improve the clarity of the paper. MJC would like to thank the Cecil King Foundation and the London Mathematical Society for a Cecil King Travel Scholarship that funded visits to the University of Wisconsin-Madison, the University of Washington, and Cornell University. QL would like to thank Vice Chancellor for Research and Graduate Education, DMS-2308440 and ONR-N000142112140. RVR would like to thank the Shanahan Family Foundation for support. AT work was partially supported by the NSF DMS-1952757, DMS-2045646, a Simons Mathematical Fellowship, and ONR-N000142312729. 
\end{acknowledgements}

{\small
\noindent\textbf{Data Availability} All data generated or analyzed during this
study is available upon request.

\noindent{}A preprint of this work can be found at~\cite{colbrook2023beyond}.

\noindent\textbf{Conflict of interest} The authors declare that they have no conflict of interest.}

\bibliographystyle{spmpsci}      
\bibliography{bib_file2}   

\end{document}